\documentclass[11pt,onecolumn]{article}
\usepackage{CJK}
\usepackage{latexsym}
\usepackage{amssymb, graphicx, amsmath}
\usepackage{booktabs}
\usepackage{cases}
\usepackage{color}
\definecolor{myblue}{rgb}{0,0,0.5}
\definecolor{mygreen}{rgb}{0,0.5,0}
\definecolor{myred}{rgb}{0.5,0,0}

\usepackage[normalem]{ulem}
\usepackage{enumerate}

\usepackage{amsmath,amsthm,amssymb}
\usepackage{graphicx}
\usepackage{subfigure}
\usepackage{caption}
\usepackage{epstopdf}
\usepackage{multirow}
\usepackage{hyperref}
\usepackage{url}

\newcommand{\D}{\mathbb{D}iv \;}
\usepackage{listings}
\usepackage{float}

\usepackage{threeparttable}

\newcommand{\nn}{\nonumber}
\newcommand{\RNum}[1]{\uppercase\expandafter{\romannumeral #1\relax}}

\def \[{\begin{equation}}
\def \]{\end{equation}}

\newtheorem{theorem}{Theorem}[section]

\newtheorem{lemma}{Lemma}[section]

\newtheorem{remark}{Remark}

\newtheorem{proposition}{Proposition}[section]

\begin{CJK}{GBK}{song}
\gdef\cdefinition{定义\,}
\gdef\clemma{引理\,}
\gdef\ctheorem{定理\,}
\end{CJK}

\begin{document}
\begin{CJK*}{GBK}{song}

\begin{center}

{\LARGE  \bf  A generalized primal-dual algorithm with improved convergence condition for saddle point problems}\\

\bigskip
\medskip

 {\bf Bingsheng He}\footnote{\parbox[t]{16cm}{
 Department of Mathematics,  Nanjing University, China.
  This author was supported by the NSFC Grant 11871029. Email: hebma@nju.edu.cn}}
 \quad
  {\bf Feng Ma}\footnote{\parbox[t]{16.0cm}{
High-Tech Institute of Xi'an, Xi'an, 710025, Shaanxi, China. This author was supported by the NSFC Grant
 12171481. Email:
 mafengnju@gmail.com}}
   \quad
 {\bf Shengjie Xu}\footnote{\parbox[t]{16cm}{
Department of Mathematics,   Southern  University of Science and Technology, Shenzhen, China. This author was supported by the Guangdong Basic and Applied Basic Research Foundation of China under grant No. 2018A0303130123.  Email: xsjnsu@163.com
  }}
  \quad
 {\bf Xiaoming Yuan}\footnote{\parbox[t]{16cm}{
 Department of Mathematics, The University of Hong Kong, Hong Kong. This author was supported  by the General Research Fund from Hong Kong Research Grant Council: 12302318. Email:  xmyuan@hku.hk
  }}

\bigskip


    October 19, 2021

\end{center}

\bigskip

{\small

\parbox{0.95\hsize}{

\hrule

\medskip

{\bf Abstract.} We generalize the well-known primal-dual algorithm proposed by Chambolle and Pock for saddle point problems, and improve the condition for ensuring its convergence. The improved convergence-guaranteeing condition is effective for the generic setting, and it is shown to be optimal. It also allows us to discern larger step sizes for the resulting subproblems, and thus provides a simple and universal way to improve numerical performance of the original primal-dual algorithm. In addition, we present a structure-exploring heuristic to further relax the convergence-guaranteeing condition for some specific saddle point problems, which could yield much larger step sizes and hence significantly better performance. Effectiveness of this heuristic is numerically illustrated by the classic assignment problem.

\medskip


 \medskip

  \hrule

  }}

\section{Introduction}

A fundamental mathematical model is the saddle point problem:
\begin{equation}\label{Min-Max}
  \min_{x\in \cal{X}} \max_{y\in \cal{Y}} {\cal L}(x,y) := f(x) -  y^TAx - g(y),
\end{equation}
where ${\cal X} \subseteq \Re^n$ and ${\cal Y} \subseteq \Re^m$ are closed convex sets, $f:{\cal X}\to \Re$ and $g:{\cal Y}\to \Re$ are proper convex but not necessarily smooth functions, and $A\in\Re^{m\times n}$ is a given matrix. The saddle point problem \eqref{Min-Max} includes many special cases such as the optimality condition of the canonical convex programming problem with linear constraints (see Section \ref{sec-connection}), scientific computing problems (see \cite{AHU}), and particularly a number of variational models arising in image reconstruction problems
 (see \cite{CHPock,CP-Acta,EZC,ZhuChan}).

To solve the saddle point problem (\ref{Min-Max}), the primal-dual algorithm proposed by Chambolle and Pock in \cite{CHPock} is influential. Its iterative scheme reads as
 \begin{subequations} \label{C-P-a}
\begin{numcases}{}
\label{C-P-x-0}  x^{k+1} =\arg\min \big\{ {\cal L}(x,y^k) + \frac{r}{2}\|x-x^k\|^2 \;|\; x\in {\cal X} \big\},\\[0.1cm]
\label{C-P-x-0-bar} \bar{x}^{k+1}= x^{k+1} +\alpha(x^{k+1}-x^k),\\[0.1cm]
\label{C-P-y-0}
     y^{k+1} = \arg\max \big\{ {\cal L}(\bar{x}^{k+1}, y) - \frac{s}{2}\|y-y^k\|^2  \;|\; y\in {\cal Y} \big\},
\end{numcases}
\end{subequations}
where $\alpha \in [0,1]$ is an extrapolation parameter, $r>0$ and $s>0$ are the regularization parameters for the proximal regularization terms in the decomposed subproblems (\ref{C-P-x-0}) and (\ref{C-P-y-0}), respectively. We refer to, e.g., \cite{CP-MP,Condat,EZC,HeMaYuan2017,HeYouYuan,MP2016,PockC,Zhang}, for discussions and generalizations on the algorithm (\ref{C-P-a}) from different perspectives. In particular, the algorithm (\ref{C-P-a}) is reduced to the Arrow-Hurwicz method in \cite{AHU} when $\alpha=0$, and it has been reemphasized as the primal-dual hybrid gradient method (PDHG) since the work \cite{ZhuChan}. As studied in various literatures such as \cite{BR2012,CHPock,EZC,HeYouYuan}, convergence of the PDHG can be established only when some more restrictive conditions are additionally assumed, such as strong convexity of the functions or some demanding requirements on the step sizes. Most recently, it was shown in \cite{HXY-AH} that, for the algorithm (\ref{C-P-a}) with $\alpha=0$, convergence is not guaranteed in the generic setting without these additional conditions. Moreover, to the best of our knowledge, convergence of the algorithm (\ref{C-P-a}) with $\alpha\in (0,1)$ is still unknown. Then, the most popular choice is the remaining case with $\alpha=1$ for the primal-dual algorithm (\ref{C-P-a}), and the resulting scheme is
\begin{subequations} \label{C-P}
\begin{numcases}{\hbox{(CP-PPA)}}
\label{C-P-0x}  x^{k+1} =\arg\min \big\{ {\cal L}(x,y^k) + \frac{r}{2}\|x-x^k\|^2  \;|\;  x\in {\cal X} \big\},\\[0.1cm]
\label{C-P-0-bar} \bar{x}^{k+1}= 2x^{k+1} -x^k,\\[0.1cm]
\label{C-P-0y}  y^{k+1} = \arg\max \big\{ {\cal L}(\bar{x}^{k+1}, y) - \frac{s}{2}\|y-y^k\|^2   \;|\;   y\in {\cal Y} \big\}.
\end{numcases}
\end{subequations}
In \cite{CHPock}, the parameters $r>0$ and $s>0$ are required to satisfy the condition
\begin{equation}\label{Assum-CP}
 r\cdot s >\rho(A^T\!A)
\end{equation}
to theoretically ensure the convergence of \eqref{C-P}, where $\rho(\cdot)$ denotes the spectrum of a matrix. For (\ref{C-P}), it was shown by He and Yuan in \cite{HeYuanSIAMIS} that it can be interpreted as an application of the classic proximal point algorithm (PPA) proposed in \cite{Martinet70,Rock76B}, and the condition (\ref{Assum-CP}) essentially ensures the positive definiteness of the underlying matrix which is used to define the corresponding proximal term. We thus call the algorithm \eqref{C-P} ``CP-PPA" for short. The CP-PPA (\ref{C-P}) is more preferable than the general primal-dual algorithm (\ref{C-P-a}) with other choices of $\alpha$, because of its better theoretical properties such as the guaranteed convergence and the mentioned PPA interpretation, as well as its more attractive numerical performance widely shown in the literatures such as \cite{CC2010,CHPock,CP-Acta,HeYuanSIAMIS}.

%

Numerical performance of the CP-PPA (\ref{C-P}) certainly depends on the choices of $r$ and $s$. Indeed, $r$ and $s$ determine the step sizes for solving the subproblems (\ref{C-P-0x}) and (\ref{C-P-0y}), respectively. To discern more appropriate choices of $r$ and $s$ for a specific problem, it is necessary to first investigate their theoretical role in ensuring convergence of the CP-PPA (\ref{C-P}) for the generic setting of (\ref{Min-Max}). An apparent effort is to consider relaxing the condition (\ref{Assum-CP}) and hence gain the possibility of enlarging the step sizes, while convergence of the CP-PPA (\ref{C-P}) can be still kept theoretically. Our main purpose is to propose a generalized version of the CP-PPA (\ref{C-P}) and improve the convergence condition (\ref{Assum-CP}) in the generic setting of (\ref{Min-Max}). More specifically, we generalize the CP-PPA (\ref{C-P}) as

\begin{subequations} \label{H-Y}
\begin{numcases}{\hbox{(Generalized CP-PPA)}\quad}
\label{H-Y-x}  x^{k+1} =\arg\min \big\{ {\cal L}(x,y^k) + \frac{r}{2}\|x-x^k\|^2   \;|\;  x\in {\cal X} \big\}, \\[0.1cm]
\label{H-Y-x-2} \bar{x}^{k+1}    = x^{k+1}   + \alpha(x^{k+1}-x^k),  \\[0.1cm]
\label{H-Y-y} \bar{y}^{k+1} = \arg\max \big\{ {\cal L}(\bar{x}^{k+1},y) - \frac{s}{2}\|y-y^k\|^2  \;|\;   y\in {\cal Y} \big\},\\
\label{H-Y-y-2} y^{k+1}=\bar{y}^{k+1}-(1-\alpha)\frac{1}{s}A({x}^{k+1}-x^k),
\end{numcases}
\end{subequations}
in which $\alpha\in [0,1]$, and then improve the condition (\ref{Assum-CP}) as
\begin{equation} \label{Assum}
  r\cdot s > (1-\alpha +\alpha^2)  \rho(A^T\!A).
\end{equation}
It is trivial to see that the generalized CP-PPA (\ref{H-Y}) with (\ref{Assum}) includes the CP-PPA (\ref{C-P}) with (\ref{Assum-CP}) as the special case of $\alpha=1$. It is also clear that $(1-\alpha +\alpha^2)\leq1$ for $\alpha\in[0,1]$\footnote{\parbox[t]{16cm}{Indeed, $\alpha$ could be any number to carry out our theoretical analysis to be presented, but we are only interested in $\alpha \in [0,1]$ because a larger lower bound of $r \cdot s$ is less meaningful.}}. Hence, the theoretical lower bound of $r\cdot s$ determined by (\ref{Assum-CP}) is improved from scratch and it becomes possible to enlarge the step sizes for the subproblems (\ref{C-P-0x}) and (\ref{C-P-0y}) in lieu with the new condition (\ref{Assum}). We will prove in Section \ref{sec-connection-1} that the generalized CP-PPA (\ref{H-Y}) coincides with the CP-PPA (\ref{C-P}) for any $\alpha$ if the saddle point problem (\ref{Min-Max}) is formed by the canonical convex programming problem with linear equality constraints (see (\ref{Problem-LC})). Obviously, $\alpha=1/2$ is the best choice to yield the minimal value of $(1-\alpha +\alpha^2)= 0.75$. Hence, if we choose $\alpha=1/2$, then the improved condition (\ref{Assum}) becomes
\begin{equation} \label{Assum-0.75}
 r\cdot s > 0.75  \rho(A^T\!A),
\end{equation}
and this improvement is effective for the generic setting of the canonical convex programming problem with linear equality  constraints (see(\ref{Problem-LC})). More specifically, the improved condition (\ref{Assum-0.75}) theoretically enlarges the range of $r\cdot s$ so that better choices of $r$ and $s$ can be empirically probed to accelerate the CP-PPA (\ref{C-P}) for specific problems. We will verify the acceleration of (\ref{Assum-0.75}) in Section \ref{Sec-num} by some concrete applications. Indeed, we will prove in Section \ref{sec-connection-2} that the condition (\ref{Assum-0.75}) is optimal for the CP-PPA (\ref{C-P}) in the sense that there always exists an example for which the CP-PPA (\ref{C-P}) is divergent if the constant $0.75$ is replaced by any other smaller positive number in (\ref{Assum-0.75}). It is certainly non-meaningful to expect that the numerical acceleration could be dramatic, given that the lower bound is refined by $25\%$ from (\ref{Assum-CP}) to (\ref{Assum-0.75}) and the refinement is  effective for the generic canonical convex programming problem with linear equality  constraints. Note that the generalized CP-PPA \eqref{H-Y} does not entail any substantial additional computation because there is no need to compute the multiplications $Ax^k$ and $A{x}^{k+1}$ in the step \eqref{H-Y-y-2}; they are available when the subproblems \eqref{H-Y-x} and \eqref{H-Y-y} are solved. It is also worth recalling that convergence of the algorithm (\ref{C-P-a}) is unknown for $\alpha \in (0,1)$, while convergence of the generalized CP-PPA (\ref{H-Y}) can be rigorously ensured (see Section \ref{Sec-con}), along with its provable worst-case convergence rate measured by iteration complexity (see Section \ref{sec-convrate}), for any $\alpha \in [0,1]$.

The other purpose of this work is to further relax the condition (\ref{Assum-CP}) heuristically for some specific cases of the saddle point problem (\ref{Min-Max}), and improve the performance of the CP-PPA (\ref{C-P}) empirically by discerning even larger step sizes. This consideration is necessary if the corresponding $\rho(A^T\!A)$ is too large and thus $r\cdot s$ is restricted by a large lower bound in either (\ref{Assum-CP}) or (\ref{Assum}). For this case, it is necessary for implementing the CP-PPA (\ref{C-P}) to avoid tiny step sizes. Generically, a convergence-guaranteeing condition for certain algorithm is sufficient, but not necessary, and thus such a condition is generally too conservative despite that the sufficiency favors theoretical analysis and mathematical rigor. This means such a sufficient condition can and should be relaxed to some extent when an algorithm is implemented, though mathematical rigor may not be maintained. Of course, relaxing such a theoretically sufficient condition should not be random, but be niche targeting. For the condition (\ref{Assum-CP}), one possible strategy is exploring the structure of $A^T\!A$ per se meticulously for a given specific problem, and combining it with the generic-purpose condition (\ref{Assum-CP}). Indeed, such a heuristical study for a specifically given problem is complementary to the methodological and theoretical study in the generic setting. We will show in Section \ref{sec-exp} that this heuristic idea works extremely well for the classic assignment problem when it is relaxed as a convex programming problem with linear equality constraints.

%
%

\section{Preliminaries}\label{sec-pre}
\setcounter{equation}{0}
\setcounter{remark}{0}

In this section, we summarize some preliminaries and recall some known results for further analysis.

\subsection{Lemma}

The following lemma is basic and will be frequently used in our analysis. Its proof can be found in, e.g., \cite{Beck}.

\begin{lemma} \label{CP-TF}
\begin{subequations} \label{CP-TF0}
 Let ${\cal Z} \subseteq \Re^n$ be a closed convex set, $\theta(z)$ and $f(z)$ be convex functions.
   If $f$ is differentiable on an open set which contains ${\cal Z}$, and the solution set of the minimization problem
   $$
   \min\{\theta(z) + f(z) \; |\; z\in {\cal Z}\}
   $$ is nonempty, then we have
  \[  \label{CP-TF1}   z^*  \in \arg\min \{  {\theta}(z) + f(z)  \; | \;  z\in {\cal Z}\}
     \]
if and only if
\[\label{CP-TF2}
      z^*\in {\cal Z}, \quad   \theta(z) - \theta(z^*) + (z-z^*)^T\nabla f(z^*) \ge 0, \quad \forall\, z\in {\cal Z}.
      \]
\end{subequations}
\end{lemma}

\subsection{Variational inequality reformulation}

Similar as our previous work \cite{HeYuanSIAMIS}, theoretical analysis for the generalized CP-PPA (\ref{H-Y}) will be based on the variational inequality (VI) reformulation of the saddle point problem (\ref{Min-Max}). Note that a solution point of (\ref{Min-Max}), denoted by $(x^*,y^*)\in {\cal X}\times {\cal Y}$, satisfies the following inequalities:
$${\cal L}_{y\in {\cal Y}}(x^*,y) \le {\cal L}(x^*,y^*) \le {\cal L}_{x\in {\cal X}}(x,y^*).$$
That is, for a solution point $(x^*, y^*)$ of (\ref{Min-Max}), we have
$$
    \left\{ \begin{array}{l}
     x^*  \in \arg\min \big\{ {\cal L}(x,y^*)  \;|\;  x\in {\cal X} \big\},   \\[0.1cm]
     y^*\in \arg\max \big\{{\cal L}(x^*,y) \;|\; y\in {\cal Y}\big\}.
        \end{array} \right.
       $$
Recall that the functions $f$ and $g$ in (\ref{Min-Max}) are not necessarily smooth. According to Lemma \ref{CP-TF}, $(x^*, y^*)$ satisfies the following inequalities:
$$  
    \left\{ \begin{array}{lll}
     x^*\in {\cal X},  &   f(x) -  f(x^*) + (x-x^*)^T(- A^Ty^*) \ge 0, & \forall\; x\in {\cal X}, \\[0.1cm]
     y^*\in {\cal Y},  &     g(y) - g(y^*)  + (y-y^*)^T(Ax^*)\ge 0,  &  \forall \;  y\in {\cal Y}.
        \end{array} \right.
       $$
These two inequalities can be compactly written as the following VI:
\begin{subequations}\label{VI}
\[  \label{VI-S}
 \hbox{VI}(\Omega, F,\theta): \quad  u^*\in \Omega, \quad \theta(u) -\theta(u^*) + (u-u^*)^T F(u^*) \ge 0, \quad \forall \,  u\in
      \Omega,
     \]
where
\[ \label{Notation-uFO}
    u = \left(\begin{array}{c}
                     x\\
                   y \end{array} \right),
  \quad \theta(u) = f(x) + g(y), \quad
    F(u) =\left(\begin{array}{c}
     - A^Ty \\
     Ax \end{array} \right)  \quad\hbox{and}\quad \Omega = {\cal X} \times {\cal Y}.
    \]
   \end{subequations}
Note that the operator $F$ given in (\ref{Notation-uFO}) is monotone, because
 \[\label{EQF}
    (u- v)^T(F(u) -F(v))\equiv 0, \quad \forall\; u,v \in \Re^{(n+m)}. \]
We denote by $\Omega^*$ the solution set of the VI \eqref{VI} throughout.

\section{Connection between (\ref{C-P}) and (\ref{H-Y})}\label{sec-connection}
\setcounter{equation}{0}
\setcounter{remark}{0}

In this section, we take a closer look at the generalized CP-PPA (\ref{H-Y}) and its connection to the CP-PPA (\ref{C-P}) in the context of the canonical convex programming problem with linear equality constraints. That is, we consider
\begin{equation}\label{Problem-LC}
   \min \{ f(x)   \;|\;  Ax=b, \,  x\in {\cal X} \},
\end{equation}
where $f: \Re^n\to \Re$ is a proper convex function, ${\cal X}\subseteq\Re^n$ is a closed convex set, $A\in\Re^{m \times n}$ is a given matrix, and $b\in \Re^m$. The Lagrangian function of \eqref{Problem-LC} is
$$
  L(x,y) = f(x) -  y^T(Ax-b)=f(x) -  y^TAx-(-b^Ty),
$$
with $y \in \Re^m$ the Lagrange multiplier. Then, $(x^*,y^*)$ is called a saddle point of $ L(x,y)$ if the following inequalities hold:
$$    L_{y\in\Re^m}(x^*,y) \le L(x^*,y^*) \le L_{x\in {\cal X}}(x,y^*). $$
It is clear that finding a saddle point of  $L(x,y)$ is a special case of \eqref{Min-Max} with $g(y) = -b^Ty$ and ${\cal Y} = \Re^m$.

\subsection{Equivalence of (\ref{C-P}) and (\ref{H-Y}) for (\ref{Problem-LC})}\label{sec-connection-1}

When the CP-PPA \eqref{C-P} is applied to the specific model \eqref{Problem-LC}, it follows from $g(y) = -b^Ty$ that
 $$    L( [2x^{k+1}  -x^k], y) = f(2x^{k+1}  -x^k) -  y^TA(2x^{k+1}  -x^k)  + b^Ty.  $$
Because  ${\cal Y}=\Re^m$, the solution of the subproblem \eqref{C-P-0y} satisfies
 $$   -A(2x^{k+1}  -x^k)  + b - s(y^{k+1} - y^k) = 0.   $$
Thus, the CP-PPA \eqref{C-P} for (\ref{Problem-LC}) is specified as
 \begin{subequations} \label{C-PA}
\begin{numcases}{}
\label{C-PA-x}  x^{k+1} =\arg\min \big\{ L(x,y^k) + \frac{r}{2}\|x-x^k\|^2 \;|\; x\in {\cal X} \big\},\\
\label{C-PA-l}
       y^{k+1} = y^k - \frac{1}{s}\bigl[     A(2x^{k+1} -x^k) -b\bigr].
\end{numcases}
\end{subequations}
Similarly, it is easy to see that the generalized CP-PPA \eqref{H-Y} for \eqref{Problem-LC} is specified as
 \begin{subequations} \label{H-Y-1}
\begin{numcases}{}
\label{H-Y-1-x}  x^{k+1} = \arg\min\big\{ L(x,y^k) + \frac{r}{2}\|x-x^k\|^2  \;|\;  x\in {\cal X} \big\},\\
\label{H-Y-1-y}  \bar{y}^{k+1}  = y^k - \frac{1}{s} \bigl(  A[{x}^{k+1} +\alpha(x^{k+1}-x^k)] -b   \bigr),\\
\label{H-Y-1-y1} y^{k+1} = \bar{y}^{k+1}  -(1-\alpha)\frac{1}{s}  A({x}^{k+1}-x^k).
\end{numcases}
\end{subequations}
To see the equivalence of (\ref{C-PA}) and (\ref{H-Y-1}) for the specific model (\ref{Problem-LC}), we know that
 \begin{eqnarray*}
        y^{k+1}  & \stackrel{\eqref{H-Y-1-y1}}{=}  &  \bar{y}^{k+1}  -(1-\alpha)\frac{1}{s}  A({x}^{k+1}-x^k)  \nn  \\
              & \stackrel{\eqref{H-Y-1-y}}{=} &  y^k -   \frac{1}{s} \bigl(  A[{x}^{k+1} +\alpha(x^{k+1}-x^k)] -b   \bigr)
              -(1-\alpha)\frac{1}{s}  A({x}^{k+1}-x^k)      \nn\\
              & = & y^k - \frac{1}{s}\bigl[     A(2x^{k+1} -x^k) -b\bigr].
        \end{eqnarray*}
Thus, (\ref{C-PA}) and (\ref{H-Y-1}) coincide for any $\alpha$ when the specific model (\ref{Problem-LC}) is considered.

\subsection{Optimality of (\ref{Assum-0.75}) for CP-PPA (\ref{C-P})}\label{sec-connection-2}

We take the example in \cite{HMY2020IMA}, which is a special case of the model (\ref{Problem-LC}), to show that the condition (\ref{Assum-0.75}) is optimal for the CP-PPA (\ref{C-P}) to ensure its convergence in the generic setting of (\ref{Min-Max}). Recall that the example in \cite{HMY2020IMA} is
\begin{equation}\label{counterexample}
  \min\{0\cdot x \;\;|\;\; x=0,\; x\in\Re\}.
\end{equation}
It is clear that $x=0$ is the solution point of \eqref{counterexample} and $\rho(A^T\!A)=1$ because $A=1$ for the example.

Now, we show that convergence of the CP-PPA (\ref{C-P}) for the example (\ref{counterexample}) is not guaranteed if the constant $0.75$ in (\ref{Assum-0.75}) is replaced by any other smaller positive number. Let us fix $s=1$. Then, the condition (\ref{Assum-0.75}) becomes $r>0.75$ and we need to show that the CP-PPA \eqref{C-PA} is divergent for any $0<r<0.75$. When the problem (\ref{counterexample}) is considered, it is easy to derive that the CP-PPA (\ref{C-PA}) with $s\equiv 1$ becomes
 \begin{subequations} \label{C-PA1}
\begin{numcases}{}
\label{C-PA-xx}  x^{k+1} = x^k+\frac{1}{r}y^k, \\
\label{C-PA-yy}  y^{k+1} = y^k -(2x^{k+1} -x^k).
\end{numcases}
\end{subequations}
Note that
$$y^{k+1} \overset{\eqref{C-PA-yy}}{=} y^k -(2x^{k+1} -x^k) \overset{\eqref{C-PA-xx}}{=} y^k -\big\{2(x^k+\frac{1}{r}y^k) -x^k\big\} = -x^k+(1-\frac{2}{r})y^k. $$
Hence, iterations generated by \eqref{C-PA1} can be recursively rewritten as
\begin{equation}\label{C-PA2}
  u^{k+1}=\mathcal{P}(r)u^k \quad \hbox{with} \quad \mathcal{P}(r) = \left(
                                                   \begin{array}{cc}
                                                     1 & \frac{1}{r} \\
                                                     -1 & 1-\frac{2}{r} \\
                                                   \end{array}
                                                 \right).
\end{equation}
It is clear that the matrix $\mathcal{P}(r)$  has the following two eigenvalues:
\begin{equation}\label{eigenfuc}
  \lambda_1(r)=(1-\frac{1}{r})-\sqrt{\frac{1}{r^2}-\frac{1}{r}} \quad \hbox{and} \quad  \lambda_2(r)=(1-\frac{1}{r})+\sqrt{\frac{1}{r^2}-\frac{1}{r}}.
\end{equation}
Thus, when $r=0.75$, we have $\lambda_1(0.75)=-1$ and $\lambda_2(0.75)=1/3$. Note that
$$\lambda_1^\prime(r)=\frac{1}{r^2}+\frac{\frac{2}{r^3}-\frac{1}{r^2}}{2\sqrt{\frac{1}{r^2}-\frac{1}{r}}}>0 \quad \hbox{when} \quad r\in(0,1).$$
That is, $\lambda_1(r)<-1$ if $0<r<0.75$, and this means divergence of the iterations (\ref{C-PA2}) for any $0<r<0.75$. Hence, the condition (\ref{Assum-0.75}) is optimal for the CP-PPA (\ref{C-P}) in the sense that there always exists an example such that it is divergent if the constant $0.75$ in (\ref{Assum-0.75}) is replaced by any other smaller positive number.

\section{Convergence analysis}\label{Sec-con}
\setcounter{equation}{0}
\setcounter{remark}{0}

In this section, we prove convergence of the generalized CP-PPA (\ref{H-Y}) with the improved condition (\ref{Assum}).

\subsection{Some matrices}

To simplify the notation for the convergence analysis, we first define two matrices as the following:
\begin{equation}\label{QM}
 Q= \left(\begin{array}{cc}
       r I_n &   A^T\\
    \alpha A   & sI_m \end{array}\right) \quad \hbox{and} \quad  M = \left(\begin{array}{cc}
        I_n &   0 \\
         -(1-\alpha)\frac{1}{s} A   & I_m \end{array} \right).
\end{equation}
Then, with the matrices $M$ and $Q$ defined in \eqref{QM}, we define two more matrices as
 \begin{equation}\label{HMQG}
   H:= QM^{-1} \quad \hbox{and} \quad G:= Q^T + Q -M^THM,
 \end{equation}
and show that the condition (\ref{Assum}) ensures their positive definiteness.

\begin{proposition}\label{propos}
The condition \eqref{Assum} ensures the positive definiteness of both the matrices $H$ and $G$ defined in \eqref{HMQG}.
\end{proposition}

\begin{proof}
It follows from \eqref{QM} and \eqref{HMQG} that
\begin{equation}\label{Matrix-H}
  H  =  QM^{-1}
       = \left(\begin{array}{cc}
       r I_n  &   A^T\\
       \alpha A  & sI_m \end{array} \right)  \left(\begin{array}{cc}
        I_n &  0  \\
         (1-\alpha)\frac{1}{s} A   & I_m \end{array} \right)  =
         \left(\begin{array}{cc}
         r I_n + (1-\alpha)\frac{1}{s} A^T\!A &   A^T\\[0.1cm]
         A   & sI_m \end{array} \right),
\end{equation}
and
\begin{eqnarray}  \label{Matrix-G}
  G  &=  &    Q^T +Q -  M^THM   =   Q^T +Q -  M^TQ   \nn    \\[0.1cm]
       & = &  \left(\begin{array}{cc}
       2 r I_n &  (1+\alpha)  A^T\\
         (1+\alpha)   A  & 2sI_m \end{array} \right) -  \left(\begin{array}{cc}
        I_n &  - (1-\alpha)\frac{1}{s} A^T  \\
         0  & I_m \end{array} \right)\left(\begin{array}{cc}
        r I_n &   A^T\\
        \alpha A  &  sI_m \end{array} \right)   \nn \\[0.1cm]
        & = &      \left(\begin{array}{cc}
       2 r I_n &  (1+\alpha)  A^T\\
         (1+\alpha)   A  & 2sI_m \end{array} \right)  -
          \left(\begin{array}{cc}
        r I_n  - \alpha(1-\alpha) \frac{1}{s} A^T\!A&  \alpha  A^T\\
         \alpha   A  & sI_m \end{array} \right)    \nn \\
          & =&  \left(\begin{array}{cc}
        r I_n  + \alpha(1-\alpha) \frac{1}{s} A^T\!A&    A^T\\
            A  & sI_m \end{array} \right).
       \end{eqnarray}
Let us define a nonsingular matrix as
 $$    C=  \left(\begin{array}{cc}
               I_n   & 0  \\
                 -\frac{1}{s}  A  &   I_m  \end{array} \right).   $$
Then, it is clear that
\begin{eqnarray*}
C^T  HC  & =  &   \left(\begin{array}{cc}
                   I_n   &     -\frac{1}{s}  A^T \\
                 0    &  I_m \end{array} \right)
                   \left(\begin{array}{cc}
       r I_n + (1-\alpha)\frac{1}{s} A^T\!A &   A^T\\
         A   & sI_m \end{array} \right)    \left(\begin{array}{cc}
                   I_n  & 0 \\
           -\frac{1}{s}  A   &I_m \end{array} \right)    \nn \\
           & = &    \left(\begin{array}{cc}
       r I_n -\frac{\alpha}{s} A^T\!A &   0\\[0.1cm]
         0  & sI_m \end{array} \right).
\end{eqnarray*}
Hence, it follows from the law of inertia that the matrix $H$ is positive definite if and only if  $r\cdot s\cdot I_n  \succ \alpha A^T\!A $. Since $1-\alpha+\alpha^2\geq\alpha$ for any $\alpha\in\Re$, it is clear that the matrix $H\succ0$ when $r\cdot s > (1-\alpha +\alpha^2)\rho(A^T\!A)$. Similarly,  because
       \begin{eqnarray*}
       C^T  GC  & =   &   \left(\begin{array}{cc}
                   I_n   &     -\frac{1}{s}  A^T \\
                 0    &  I_m \end{array} \right)
       \left(\begin{array}{cc}
        r I_n  + \alpha(1-\alpha) \frac{1}{s} A^T\!A&    A^T\\
            A  & sI_m \end{array} \right)
                    \left(\begin{array}{cc}
                   I_n  & 0 \\
           -\frac{1}{s}  A   &I_m \end{array} \right)    \nn \\
           & = &    \left(\begin{array}{cc}
       r I_n -\frac{1}{s}(1-\alpha +\alpha^2) A^T\!A &   0\\[0.1cm]
         0  & sI_m \end{array} \right),
         \end{eqnarray*}
the matrix $G$ is positive definite if and only if $r\cdot s \cdot I_n \succ  (1-\alpha +\alpha^2) A^T\!A$.
 It is obvious that $G\succ0$ when $r\cdot s > (1-\alpha +\alpha^2)\rho(A^T\!A)$. This completes the proof of this proposition.
 \end{proof}

\subsection{Prediction-correction representation}

Then, we rewrite the generalized CP-PPA \eqref{H-Y} as a prediction-correction framework, and this framework further helps us conduct the convergence analysis with easier notation. More specifically, for a given $u^k =(x^k, y^k)$, the iterate $u^{k+1}=(x^{k+1}, y^{k+1})$ generated by the generalized CP-PPA (\ref{H-Y}) with (\ref{Assum}) can be represented as the following two steps. Recall that the matrix $M$ in (\ref{C-P-C}) is defined in (\ref{QM}).

\begin{center}\fbox{
 \begin{minipage}{16cm}
 \smallskip
 \begin{subequations} \label{C-P-P}
\begin{numcases}{\hbox{\bf(Prediction)}}
\label{C-PP-x}  \tilde{x}^{k}  = \arg\min \big\{ {\cal L}(x,y^k) + \frac{r}{2}\|x-x^k\|^2    \;|\;   x\in {\cal X} \big\},\\[0.1cm]
\label{C-PP-y}  \tilde{y}^{k}  = \arg\max \big\{ {\cal L}([\tilde{x}^k +\alpha(\tilde{x}^k-x^k)],y) - \frac{s}{2}\|y-y^k\|^2   \;|\;  y\in {\cal Y} \big\}.
\end{numcases}
\end{subequations}

\begin{equation}\label{C-P-C}
  \quad \hbox{\bf(Correction)} \quad  u^{k+1}  =  u^k  - M(u^k -\tilde{u}^k). \qquad \qquad\qquad \qquad\qquad \qquad \qquad \quad
\end{equation}
\smallskip
\vspace{-0.3cm}
 \end{minipage}}
\end{center}

We would emphasize that this prediction-correction representation is merely used for theoretical analysis, and there is no need to implement the generalized CP-PPA \eqref{H-Y} by following this prediction-correction framework.

\subsection{Convergence}

Now, we take advantage of the prediction-correction representation (\ref{C-P-P})-(\ref{C-P-C}), and prove convergence of the generalized CP-PPA (\ref{H-Y}). The following lemma characters the difference of the predictor $\tilde{u}^k$ represented by (\ref{C-P-P}) from a solution point of the VI (\ref{VI}).

\begin{lemma}  Let $\tilde{u}^k = (\tilde{x}^k, \tilde{y}^k)$ be the predictor represented  by \eqref{C-P-P} with given $u^k =(x^k, y^k)$. Then, we have
\begin{equation}\label{Pred1}
  \tilde{u}^k\in\Omega, \quad \theta(u) -\theta(\tilde{u}^{k}) + (u-\tilde{u}^{k})^T F(\tilde{u}^{k}) \ge (u-\tilde{u}^{k})^T  Q(u^k  -\tilde{u}^{k}), \quad  \forall \; u\in  \Omega,
\end{equation}
where the matrix $Q$ is defined in \eqref{QM}.
\end{lemma}
\begin{proof}
For the subproblem \eqref{C-PP-x}, it follows from Lemma \ref{CP-TF} that
 $$ \tilde{x}^{k}\in {\cal X},   \quad   f(x) -f(\tilde{x}^{k}) + (x-\tilde{x}^{k})^T \{ -A^Ty^k + r(\tilde{x}^{k}-x^k) \} \ge 0, \quad  \forall \; x\in {\cal X}, $$
which can be further rewritten as
\begin{equation}\label{C-PP-x1}
  \tilde{x}^{k}\in {\cal X},   \quad  f(x) -f(\tilde{x}^{k}) + (x-\tilde{x}^{k})^T \{ -A^T\tilde{y}^k  + r(\tilde{x}^{k}-x^k) + A^T(\tilde{y}^k  -y^k) \} \ge 0, \quad  \forall \; x\in {\cal X}.
\end{equation}
Similarly, for the subproblem \eqref{C-PP-y}, we have
\begin{equation}\label{C-PP-l1}
  \tilde{y}^k  \in {\cal Y},   \quad  g(y) -g(\tilde{y}^{k}) +  (y-\tilde{y}^k )^T \{   A[\tilde{x}^{k} + \alpha(\tilde{x}^k -x^k)]  + s(\tilde{y}^k  -y^k) \} \ge 0, \quad  \forall \; y\in {\cal Y}.
\end{equation}
Adding \eqref{C-PP-x1} and \eqref{C-PP-l1}, we have $\tilde{u}^k=(\tilde{x}^{k},\tilde{y}^k )\in {\cal X} \times {\cal Y}$ such that
\begin{eqnarray*} 
   \lefteqn{ f(x)+g(y) - (f(\tilde{x}^k)+g(\tilde{y}^k))
     + \left(\begin{array}{c}
    x- \tilde{x}^{k} \\
     y - \tilde{y}^k
     \end{array}\right)^T
     \left\{
     \left(\begin{array}{c}
      - A^T\tilde{y}^k  \\
       A\tilde{x}^{k}
     \end{array}\right) \right. }\nn \\
     & &\qquad \qquad  \qquad \qquad  + \left.
  \left(\begin{array}{c}
      r(\tilde{x}^{k} -x^k)  + A^T(\tilde{y}^k  -y^k)\\
      \alpha A(\tilde{x}^k-x^k)   +  s(\tilde{y}^k  -y^k)
     \end{array}\right)  \right\} \ge 0, \quad \forall \; (x,y)\in
      {\cal X} \times {\cal Y}.
  \end{eqnarray*}
Using the notation in \eqref{VI} and the matrix $Q$ defined in \eqref{QM}, the assertion of this lemma follows immediately.
\end{proof}

The following lemma refines the right-hand side of the inequality (\ref{Pred1}), and it enables us to use the predefined matrices $H$ and $G$ in \eqref{HMQG} to quantify the difference of $\tilde{u}^k$ from a solution point of the VI (\ref{VI}) by quadratic terms.

\begin{lemma} \label{THM-HauptA}
Let $u^k$ be a given  vector. The predictor $\tilde{u}^k $ and the corrector  $u^{k+1}$  are represented by \eqref{C-P-P} and \eqref{C-P-C}, respectively. Then, we have
 \begin{equation}\label{THM-H-A}
  \theta(u) - \theta(\tilde{u}^k) + (u- \tilde{u}^k)^TF(u)  \ge  \frac{1}{2}\bigl(\|u-u^{k+1}\|_H^2 -\|u-u^k \|_H^2  \bigr) + \frac{1}{2}\|u^k -\tilde{u}^k\|_G^2, \;\;  \forall \; u\in \Omega,
 \end{equation}
 where the matrices $H$ and $G$ are defined in \eqref{HMQG}.
\end{lemma}
\begin{proof}
According to \eqref{C-P-C} and $Q=HM$ (see \eqref{HMQG}), the right-hand side of \eqref{Pred1} can be rewritten as
 $$ (u-\tilde{u}^{k})^T Q(u^k  -\tilde{u}^{k}) = (u-\tilde{u}^{k})^T  HM(u^k  -\tilde{u}^{k}) \overset{\eqref{C-P-C}}{=}  (u -\tilde{u}^{k})^T  H(u^k  -{u}^{k+1}).$$
 It follows from the monotonicity of $F$ (see \eqref{EQF}) that
 $$ \theta(u) - \theta(\tilde{u}^k) + (u- \tilde{u}^k )^T F(u)\equiv \theta(u) - \theta(\tilde{u}^k) + (u- \tilde{u}^k )^T F(\tilde{u}^k).$$
The inequality \eqref{Pred1} is thus equivalent to
\begin{equation}\label{LEM-M1-1}
\theta(u) - \theta(\tilde{u}^k) +  (u- \tilde{u}^k )^T F(u)  \ge  (u- \tilde{u}^k)^T H(u^k-u^{k+1}),  \quad  \forall \; u\in \Omega.
\end{equation}
Applying the identity
$$(a-b)^TH(c-d) = \frac{1}{2} \big\{\|a-d\|_H^2 -\|a-c\|_H^2 \big\} + \frac{1}{2} \big\{\|c-b\|_H^2 -\|d-b\|_H^2 \big\},$$
to the right-hand side of \eqref{LEM-M1-1} with $ a=u$, $b=\tilde{u}^k$, $c=u^k$ and $d=u^{k+1}$, we have
\begin{equation}\label{LEM-M1-2}
 (u-\tilde{u}^k)^T H(u^k-u^{k+1}) = \frac{1}{2} \bigl\{ \|u-u^{k+1}\|_H^2  -  \|u-u^k \|_H^2 \bigr\} + \frac{1}{2}\big\{\|u^k -\tilde{u}^k\|_H^2  -  \|u^{k+1}-\tilde{u}^k\|_H^2\big\}.
\end{equation}
For the second part of the right-hand side of \eqref{LEM-M1-2}, because $Q=HM$ and $2u^TQu=u^T(Q^T+Q)u$, we have
\begin{eqnarray}\label{LEM-M1-3}
 \|u^k -\tilde{u}^k\|_H^2 - \|u^{k+1}-\tilde{u}^k\|_H^2 & \stackrel{\eqref{C-P-C}}{=}  &  \|u^k -\tilde{u}^k\|_H^2 - \|(u^k -\tilde{u}^k) - M (u^k -\tilde{u}^k)\|_H^2  \nonumber \\
  & = & 2 (u^k -\tilde{u}^k)^TH M (u^k -\tilde{u}^k) - (u^k -\tilde{u}^k)^TM ^THM (u^k -\tilde{u}^k) \nonumber \\
  & =  & (u^k -\tilde{u}^k)^T ( Q^T +Q - M^THM)(u^k -\tilde{u}^k) \nn \\
  & \stackrel{\eqref{HMQG}}{=}  &   \|u^k-\tilde{u}^k\|_G^2.
\end{eqnarray}
Then, substituting \eqref{LEM-M1-2} and \eqref{LEM-M1-3} into \eqref{LEM-M1-1}, we prove the assertion (\ref{THM-H-A}).
\end{proof}

The right-hand side of (\ref{THM-H-A}) represented by quadratic terms is easier to be operated recursively, and it helps us derive the strict contraction of the sequence $\{u^k\}$ generated by the generalized CP-PPA (\ref{H-Y}) in Theorem \ref{THM-HauptC}.


\begin{theorem} \label{THM-HauptC}
Let $u^k$ be a given vector. The predictor $\tilde{u}^k$ and the corrector  $u^{k+1}$  are represented by  \eqref{C-P-P}  and  \eqref{C-P-C}, respectively. It holds that
\begin{equation}\label{THM-H-C-0}
 \|u^{k+1} - u^*\|_H^2  \le \|u^k  - u^*\|_H^2 -\|u^k -\tilde{u}^k\|_G^2 , \quad  \forall  \; u^* \in \Omega^*,
\end{equation}
where the matrices $H$ and $G$ are defined in \eqref{HMQG}.
\end{theorem}
\begin{proof} Setting $u$ in \eqref{THM-H-A} as arbitrary $u^\ast\in\Omega^\ast$, we obtain
$$ \|u^k  - u^*\|_H^2  - \|u^{k+1} -u^*\|_H^2\ge \|u^k -\tilde{u}^k\|_G^2 +  2\big(\theta(\tilde{u}^k)- \theta(u^*)  +(\tilde{u}^k-u^*)^TF(u^\ast)\big).$$
Since $u^*$ is a solution point of the VI (\ref{VI}), we have
$$  \theta(\tilde{u}^k)- \theta(u^*) +  (\tilde{u}^k-u^*)^T F({u}^*)   \ge 0.$$
Thus, it holds that
$$ \|u^k  - u^*\|_H^2  - \|u^{k+1} -u^*\|_H^2 \ge \|u^k -\tilde{u}^k\|_G^2, \quad  \forall  \; u^* \in \Omega^*,$$
and the assertion (\ref{THM-H-C-0}) is proved.
\end{proof}

Now, we are ready to prove the global convergence of the generalized CP-PPA \eqref{H-Y} with the condition (\ref{Assum}).

\begin{theorem} \label{THM-HauptD}
The sequence $\{u^k \}$ generated by the generalized CP-PPA \eqref{H-Y} with the condition (\ref{Assum}) converges to a solution point of the saddle point problem (\ref{Min-Max}).
\end{theorem}

\begin{proof} It follows from \eqref{THM-H-C-0} that the sequence $\{u^k\}$ is bounded. Summarizing \eqref{THM-H-C-0} over $k=0,1,\ldots,\infty$, we obtain
$$\sum_{k=0}^{\infty}\|u^k -\tilde{u}^k\|_G^2\leq\|u^0 -u^\ast\|_H^2,$$
which implies
\begin{equation}\label{Contrac-uut}
  \lim_{k\to \infty}\|u^k -\tilde{u}^k\|_G=0.
\end{equation}
Therefore, the sequence $\{\tilde{u}^k\}$ is also bounded. Let $u^{\infty}$ be a cluster point of $\{\tilde{u}^k\}$ and $\{\tilde{u}^{k_j}\}$ be a subsequence converging to
$u^{\infty}$. Then, it follows from \eqref{Pred1} that
$$\tilde{u}^{k_j}\in \Omega, \quad \theta(u)-\theta(\tilde{u}^{k_j}) + (u-\tilde{u}^{k_j})^TF(\tilde{u}^{k_j}) \ge (u-\tilde{u}^{k_j})^TQ(u^{k_j}-\tilde{u}^{k_j}), \quad \forall\; u\in \Omega.$$
Since the matrix $Q$ defined in \eqref{QM} is nonsingular, it follows from the
continuity of $\theta(u)$ and $F(u)$ that
$$ u^{\infty}\in \Omega, \quad \theta(u)-\theta(u^{\infty}) + (u- u^{\infty})^T F(u^{\infty}) \ge 0, \quad \forall\; u\in \Omega. $$
This means $u^{\infty}$ is a solution point of the VI \eqref{VI}, which is also a solution point of the saddle point problem (\ref{Min-Max}). Moreover, it follows from \eqref{Contrac-uut} that $\lim_{j\rightarrow\infty}u^{k_j}=u^\infty$. Also, according to \eqref{THM-H-C-0}, we obtain
$$\|u^{k+1} - u^{\infty}\|_H \le \|u^k  - u^{\infty}\|_H,$$
which indicates that the sequence $\{u^k\}$ does not have more than one cluster point. Therefore, we have $\lim_{k\rightarrow\infty}u^k=u^\infty$ and the proof is complete.
\end{proof}

\section{Convergence rate}\label{sec-convrate}

\setcounter{equation}{0}
\setcounter{remark}{0}

In this section, we derive the worst-case $O(1/N)$ convergence rate measured by the iteration complexity for the generalized CP-PPA \eqref{H-Y} in both the ergodic and point-wise senses, where $N$ denotes the iteration counter.

\subsection{Convergence rate in the ergodic sense}

We follow our previous work \cite{HeYuan-SIAM-N} and prove the worst-case $O(1/N)$ convergence rate in the ergodic sense for the generalized CP-PPA \eqref{H-Y}. First of all, recall the characterization of the solution set of the VI  \eqref{VI} given in \cite{FP2003}.

\begin{theorem}
The solution set of the VI \eqref{VI} is convex and it can be characterized as
\begin{equation}\label{VI-FP}
  \Omega^\ast = \bigcap_{u\in\Omega}\big\{\tilde{u}\in\Omega:\; \theta(u) - \theta(\tilde{u}) + (u-\tilde{u})^TF(u)\geq0\big\}.
\end{equation}
\end{theorem}
\begin{proof}
See the proof of Theorem 2.3.5 in \cite{FP2003}, or Theorem 2.1 in \cite{HeYuan-SIAM-N}.
\end{proof}

Accordingly,  for a given accuracy $\epsilon>0$, $\tilde{u}\in\Omega$ is called an $\epsilon$-approximate solution point of the VI \eqref{VI} if it satisfies
$$\theta(\tilde{u}) - \theta(u) + (\tilde{u}-u)^TF(u)\leq\epsilon,\quad \forall \; u\in \mathcal{D}_{(\tilde{u})},$$
where $\mathcal{D}_{(\tilde{u})} = \big\{u\in\Omega\;|\;\|u-\tilde{u}\|\leq1\big\}$. Now, we show that,  after $N$ iterates generated by the generalized CP-PPA (\ref{H-Y}), we can find a point $\tilde{u}$ such that
\begin{equation}\label{VI-FP-solution}
\tilde{u}\in\Omega \quad \hbox{and}\quad  \sup_{u\in\mathcal{D}_{(\tilde{u})}}\big\{ \theta(\tilde{u}) - \theta(u) + (\tilde{u}-u)^TF(u)\big\}\leq \epsilon:=O(1/N).
\end{equation}
\begin{theorem}\label{crateth1}
Let $\{u^k\}$ be the sequence generated by the generalized CP-PPA (\ref{H-Y}) with the condition \eqref{Assum}. The predictor $\tilde{u}^k$ and the corrector  $u^{k+1}$  are represented by \eqref{C-P-P} and \eqref{C-P-C}, respectively. For any integer $N>0$, we define
\begin{equation}\label{averagep}
{\bar{u}_N} = \frac{1}{N+1} \sum_{k=0}^N \tilde{u}^k.
\end{equation}
Then, it holds that
\begin{equation}
{\bar{u}_N}\in\Omega,\quad  \theta(\bar{u}_N) -\theta(u)  + (\bar{u}_{N}-u)^T F(u) \le \frac{1}{2(N+1)}\|u-u^0\|_H^2, \quad \forall \;u\in\Omega.
\end{equation}
\end{theorem}
\begin{proof}
It follows from the inequality \eqref{THM-H-A} and the positive definiteness of the matrix $G$ that
\begin{equation}\label{crate1}
 \theta(\tilde{u}^k) -  \theta(u) + (\tilde{u}^k-u)^TF(u) \leq  \frac{1}{2}\bigl\{\|u-u^k \|_H^2 - \|u-u^{k+1}\|_H^2 \bigr\}, \quad  \forall \; u\in \Omega.
\end{equation}
Adding \eqref{crate1} over $k=0,1,\ldots,N$, we get
$$
  \sum_{k=0}^N\theta(\tilde{u}^k)  - (N+1)\theta(u) + \Big(\sum_{k=0}^N\tilde{u}^k- (N+1)u\Big)^TF(u) \leq  \frac{1}{2}\|u-u^0 \|_H^2, \quad  \forall \; u\in \Omega.
$$
Using the notation ${\bar{u}_N}$ defined in \eqref{averagep}, the above inequality can be rewritten as
\begin{equation}\label{crate2}
  \frac{1}{N+1}\sum_{k=0}^N\theta(\tilde{u}^k)  - \theta(u) + ({\bar{u}_N}- u)^TF(u) \leq  \frac{1}{2(N+1)}\|u-u^0 \|_H^2, \quad  \forall \; u\in \Omega.
\end{equation}
Since $\mathcal{X}$ and $\mathcal{Y}$ are convex sets, and $\tilde{u}^k\in\Omega$ for all $k\geq0$, we have ${\bar{u}_N}\in\Omega$.  On the other hand, $\theta$ is convex and thus we have
\begin{equation}\label{crate3}
   \theta({\bar{u}_N})\leq\frac{1}{N+1}\sum_{k=0}^N\theta(\tilde{u}^k).
\end{equation}
Substituting \eqref{crate3} into \eqref{crate2}, the assertion of this theorem follows immediately.
\end{proof}


Then, it follows from Theorem \ref{crateth1} that, with the first $N$ iterates generated by the generalized CP-PPA \eqref{H-Y},  the point $\bar{u}_N$  defined in \eqref{averagep} satisfies
$$\bar{u}_N\in\Omega \quad \hbox{and} \quad \sup_{u\in\mathcal{D}_{(\bar{u}_N)}}\big\{ \theta(\bar{u}_N) - \theta(u) + (\bar{u}_N-u)^TF(u)\big\}\leq \frac{c}{2(N+1)}=O(1/N),$$
where $\mathcal{D}_{(\bar{u}_N)} = \big\{u\in\Omega\;|\;\|u-\bar{u}_N\|\leq1\big\}$ and $c=\sup\{\|u-u^0\|_H^2 \;|\;u\in\mathcal{D}_{(\bar{u}_N)}\}$. The worst-case $O(1/N)$ convergence rate is thus established for the generalized CP-PPA (\ref{H-Y}) in the ergodic sense.

\subsection{Convergence rate in a point-wise sense}

Now, we follow our previous work \cite{HY-NM} and derive the worst-case $O(1/N)$ convergence rate in a point-wise sense for the generalized CP-PPA \eqref{H-Y}. The following lemma shows certain monotonicity of the sequence $\{\|M(u^k-\tilde{u}^k)\|_H^2\}$, and it helps us operate these terms recursively.

\begin{lemma}
Let $\{u^k\}$ be the sequence generated by the generalized CP-PPA (\ref{H-Y}) with the condition (\ref{Assum}) and recall its prediction-correction representation \eqref{C-P-P}-\eqref{C-P-C}.  For any integer $k\geq0$, it holds that
\begin{equation}\label{point-2}
 \|M(u^{k+1}-\tilde{u}^{k+1})\|_H^2\leq \|M(u^k-\tilde{u}^k)\|_H^2.
\end{equation}
\end{lemma}

\begin{proof}
Utilizing the identity $\|a\|_H^2-\|b\|_H^2=2a^TH(a-b)-\|a-b\|_H^2$
with $a=M(u^k-\tilde{u}^k)$ and $b=M(u^{k+1}-\tilde{u}^{k+1})$, we have
\begin{equation}\label{Bequaliy}
\begin{aligned}
& \|M(u^k-\tilde{u}^k)\|_H^2-\|M(u^{k+1}-\tilde{u}^{k+1})\|_H^2\\[0.1cm]
& \; = 2(u^k-\tilde{u}^k)^TM^THM\{(u^k-\tilde{u}^k)-(u^{k+1}-\tilde{u}^{k+1})\}-\|M\{(u^k-\tilde{u}^k)-(u^{k+1}-\tilde{u}^{k+1})\}\|_H^2.
\end{aligned}
\end{equation}
Let us first bound the crossing term in the right-hand side of \eqref{Bequaliy} by a quadratic term.
More specifically, setting $u=\tilde{u}^{k+1}$ in  \eqref{Pred1}, we have
\begin{equation}\label{CR1}
  \theta(\tilde{u}^{k+1})-\theta(\tilde{u}^k)+(\tilde{u}^{k+1}-\tilde{u}^k)^TF(\tilde{u}^k)\geq(\tilde{u}^{k+1}-\tilde{u}^k)^TQ(u^k-\tilde{u}^k).
\end{equation}
Rewrite the inequality \eqref{Pred1} for the $(k+1)$-th iteration and obtain
\begin{equation}\label{CR2}
  \theta(u)-\theta(\tilde{u}^{k+1})+(u-\tilde{u}^{k+1})^TF(\tilde{u}^{k+1})\geq(u-\tilde{u}^{k+1})^TQ(u^{k+1}-\tilde{u}^{k+1}),\quad \forall \; u\in\Omega.
\end{equation}
Setting $u=\tilde{u}^k$ in \eqref{CR2}, we get
\begin{equation}\label{CR3}
  \theta(\tilde{u}^{k})-\theta(\tilde{u}^{k+1})+(\tilde{u}^{k}-\tilde{u}^{k+1})^TF(\tilde{u}^{k+1})\geq(\tilde{u}^{k}-\tilde{u}^{k+1})^TQ(u^{k+1}-\tilde{u}^{k+1}).
\end{equation}
Adding \eqref{CR1} and \eqref{CR3},  and using the monotonicity of $F$  (see \eqref{EQF}), we obtain
\begin{equation}\label{noneg1}
(\tilde{u}^{k}-\tilde{u}^{k+1})^TQ\{(u^k-\tilde{u}^k)-(u^{k+1}-\tilde{u}^{k+1})\}\geq0.
\end{equation}
Moreover, adding the term
$
\{(u^k-\tilde{u}^k)-(u^{k+1}-\tilde{u}^{k+1})\}^TQ\{(u^k-\tilde{u}^k)-(u^{k+1}-\tilde{u}^{k+1})\}
$
to both sides of \eqref{noneg1} and using $u^TQu=\frac{1}{2}u^T(Q^T+Q)u$, we have
\begin{equation}\label{noneg2}
(u^k-u^{k+1})^TQ\{(u^k-\tilde{u}^k)-(u^{k+1}-\tilde{u}^{k+1})\}\geq\frac{1}{2}\|(u^k-\tilde{u}^k)-(u^{k+1}-\tilde{u}^{k+1})\|_{Q^T+Q}^2.
\end{equation}
Note that the left-hand side of \eqref{noneg2} can be rewritten as
\begin{eqnarray*}
 \lefteqn{(u^k-u^{k+1})^TQ\{(u^k-\tilde{u}^k)-(u^{k+1}-\tilde{u}^{k+1})\}} \\
   &\overset{\eqref{C-P-C}}{=}& \{M(u^k-\tilde{u}^k)\}^TQ\{(u^k-\tilde{u}^k)-(u^{k+1}-\tilde{u}^{k+1})\} \\
   &\overset{\eqref{HMQG}}{=}& (u^k-\tilde{u}^k)^TM^THM\{(u^k-\tilde{u}^k)-(u^{k+1}-\tilde{u}^{k+1})\}.
\end{eqnarray*}
Thus, we obtain
\begin{equation}\label{point-1}
  2(u^k-\tilde{u}^k)^TM^THM\{(u^k-\tilde{u}^k)-(u^{k+1}-\tilde{u}^{k+1})\}\geq \|(u^k-\tilde{u}^k)-(u^{k+1}-\tilde{u}^{k+1})\|_{Q^T+Q}^2.
\end{equation}
Furthermore, substituting \eqref{point-1} into \eqref{Bequaliy}, we have
\begin{eqnarray*}
  \lefteqn{\|M(u^k-\tilde{u}^k)\|_H^2-\|M(u^{k+1}-\tilde{u}^{k+1})\|_H^2} \\
   &=& 2(u^k-\tilde{u}^k)^TM^THM\{(u^k-\tilde{u}^k)-(u^{k+1}-\tilde{u}^{k+1})\}  -\|M\{(u^k-\tilde{u}^k)-(u^{k+1}-\tilde{u}^{k+1})\}\|_H^2 \\
   &\overset{\eqref{point-1}}{\geq}& \|(u^k-\tilde{u}^k)-(u^{k+1}-\tilde{u}^{k+1})\|_{Q^T+Q}^2-\|M\{(u^k-\tilde{u}^k)-(u^{k+1}-\tilde{u}^{k+1})\}\|_H^2 \\
   &=& \|(u^k-\tilde{u}^k)-(u^{k+1}-\tilde{u}^{k+1})\|_{G}^2,
\end{eqnarray*}
where $G$ is defined in \eqref{HMQG}. Recall that $G$ is positive definite under the condition \eqref{Assum} (see Proposition \ref{propos}). The assertion of this lemma is proved.
\end{proof}

Now, we are ready to derive the $O(1/N)$ convergence rate in a point-wise sense for the generalized CP-PPA \eqref{H-Y}.
\begin{theorem}\label{last-theorem}
Let $\{u^k\}$ be the sequence generated by the generalized CP-PPA (\ref{H-Y}) with the condition (\ref{Assum}) and recall its prediction-correction representation \eqref{C-P-P}-\eqref{C-P-C}. Then, for any positive integer $N$,  we have
\begin{equation}\label{keyin1}
  \|M(u^N-\tilde{u}^N)\|_H^2\leq\frac{1}{(N+1)c_0}\|u^0-u^\ast\|_H^2, \quad  \forall  \; u^* \in \Omega^*,
\end{equation}
where $c_0>0$ is a constant independent of $N$.
\end{theorem}
\begin{proof}
Recall Theorem \ref{THM-HauptC}. Because of the equivalence of different norms, there exists a constant $c_0>0$ such that
\begin{equation}\label{THM-H-C-01}
  \|u^{k+1} - u^*\|_H^2 \le \|u^k  - u^*\|_H^2 - c_0\|M(u^k -\tilde{u}^k)\|_H^2, \quad  \forall  \; u^* \in \Omega^*.
\end{equation}
Summarizing \eqref{THM-H-C-01} over $k=0,\ldots,N$, we have
$$\sum_{k=0}^{N}c_0\|M(u^k-\tilde{u}^k)\|_H^2\leq\|u^0-u^\ast\|_H^2, \quad  \forall  \; u^* \in \Omega^*.$$
It follows from \eqref{point-2} that the sequence $\{\|M(u^k-\tilde{u}^k)\|_H^2\}$ is monotonically non-increasing. We thus have
$$(N+1)c_0\|M(u^N-\tilde{u}^N)\|_H^2\leq\sum_{k=0}^{N}c_0\|M(u^k-\tilde{u}^k)\|_H^2\leq\|u^0-u^\ast\|_H^2, \quad  \forall  \; u^* \in \Omega^*,$$
which leads to the assertion of this theorem immediately.
\end{proof}
Let $d:=\inf\{\|u^0-u^\ast\|_H^2 \;|\; u^* \in \Omega^*\}$. It follows from Theorem \ref{last-theorem} that
$$
\|M(u^N-\tilde{u}^N)\|_H^2\leq\frac{d}{(N+1)c_0}=O(1/N).
$$
Recall the inequality \eqref{Pred1} and the fact $Q=HM$. Then, $\tilde{u}^k$ is a solution point of the VI \eqref{VI} if and only if $\|M(u^k-\tilde{u}^k)\|_H^2=0$. Hence, the assertion (\ref{keyin1}) indicates the worst-case $O(1/N)$ convergence rate in a point-wise sense for the generalized CP-PPA \eqref{H-Y}.

\begin{remark}
Recall \eqref{C-P-C} and we have $u^k-u^{k+1}=M(u^k-\tilde{u}^k)$. The assertion (\ref{keyin1}) also implies that $\|u^k-u^{k+1}\|<\epsilon$ can be used as a stopping criterion for implementing the generalized CP-PPA \eqref{H-Y}, where $\epsilon>0$ denotes the error tolerance.
\end{remark}

\section{Numerical experiments}\label{Sec-num}

\setcounter{equation}{0}
\setcounter{remark}{0}

In this section, we test some specific applications of the saddle point problem (\ref{Min-Max}) and report the numerical results to verify our theoretical assertions. In particular, we verify that the improved condition (\ref{Assum}) leads to better numerical performance for the CP-PPA (\ref{C-P}). The acceleration is moderately but universally effective for a class of problems in form of (\ref{Problem-LC}). We particularly focus on the best choice of $\alpha=0.5$ and compare it with the original choice of $\alpha=1$, i.e., (\ref{Assum-CP}). That is, we mainly compare the numerical difference of the CP-PPA (\ref{C-P}) with the conditions (\ref{Assum-CP}) and (\ref{Assum-0.75}). Our codes were written in Python 3.9 and they were implemented in a Lenovo laptop with  2.20 GHz Intel Core i7-8750H CPU and 16 GB memory.

\subsection{Basis pursuit}

We first consider the basis pursuit problem
\begin{equation}\label{BP}
  \min \big\{\|x\|_1 \;|\; Ax=b, \; x\in\Re^n\big\},
\end{equation}
where $\|x\|_1=\sum_{i=1}^n|x_i|$, $A\in\Re^{m\times n}$ ($m<n$) is a given data matrix, and $b\in\Re^m$. The basis pursuit problem \eqref{BP} plays a fundamental role in various areas such as compressed sensing, signal processing and statistical learning. We refer to, e.g., \cite{BD2009,Chen}, for some survey papers.

When the CP-PPA \eqref{C-P} is applied to \eqref{BP}, the resulting scheme is
\begin{equation}\label{CP-BP}
  \left\{
    \begin{array}{cll}
      x^{k+1} &=& \arg\min\big\{ \|x\|_1 + \frac{r}{2}\|x-(x^k+\frac{1}{r}A^Ty^k)\|_2^2  \;|\; x\in\Re^n\big\}, \\[0.2cm]
      y^{k+1} &=& y^k-\frac{1}{s}\big[A(2x^{k+1}-x^k)-b\big].
    \end{array}
  \right.
\end{equation}
Note that the $x$-subproblem in (\ref{CP-BP}) has a closed-form solution, and it is represented by the shrinkage operator defined in \cite{Chen}. Recall that the CP-PPA (\ref{C-P}) and the generalized CP-PPA (\ref{H-Y}) coincide for the model (\ref{BP}), and our purpose of testing (\ref{BP}) is to verify that the condition (\ref{Assum}) can result in numerical acceleration.

To simulate, we follow some standard way (e.g., as elucidated on \url{https://web.stanford.edu/~boyd/papers/admm/basis_pursuit/basis_pursuit_example.html}) to generate $x^\ast\in\Re^n$ randomly and $s$ coordinates of it are drawn from the uniform distribution in $[-10,10]$ and the rest are zeros, and then generate the matrix $A\in\Re^{m\times n}$ whose entries satisfying the normal distribution $\mathcal{N}(0,1)$. We set $b=Ax^\ast$, and use $x^0=\textbf{0}$ and $y^0=\textbf{0}$ as the initial point. Moreover, we take $m=n/4$ and set the sparse parameter $s=n/20$. The stopping criterion for (\ref{CP-BP}) is
$$\|u^k-u^{k-1}\| = \sqrt{\|x^k-x^{k-1}\|^2+\|y^k-y^{k-1}\|^2}<10^{-9}.$$
Since our focus is not discussing how to tune the parameters $r$ and $s$ empirically for a specific problem, but verifying acceleration of the improved theoretical lower bound of $r\cdot s$, we fix the mechanism of determining $r$ and $s$ individually (which is also experimentally probed) and test the numerical difference with different lower bound of $r \cdot s$. More precisely, they are chosen as follows subject to the only difference of the constant $1-\alpha+\alpha^2$.
\begin{itemize}
  \item CP-PPA with (\ref{Assum-CP}): $r=\sqrt{\rho(A^T\!A)}/10$ and $s=10\sqrt{\rho(A^T\!A)}$.
  \item CP-PPA with (\ref{Assum}): $r=\sqrt{(1-\alpha+\alpha^2)\rho(A^T\!A)}/10$ and $s=10\sqrt{(1-\alpha+\alpha^2)\rho(A^T\!A)}$.
  \item CP-PPA with (\ref{Assum-0.75}): $r=\sqrt{0.75\rho(A^T\!A)}/10$ and $s=10\sqrt{0.75\rho(A^T\!A)}$.
\end{itemize}

In Table \ref{Ta1}, for different cases of $n$, the iteration number (``It."), computing time in seconds (``CPU"), the objective function value (``$\|x\|_1$"), and the violation of constraints (``$\|Ax^k-b\|$") are reported. The moderate acceleration of the optimal condition (\ref{Assum-0.75}) over (\ref{Assum-CP}) is clearly seen in Table \ref{Ta1}. To see the acceleration of other values of $\alpha$ in (\ref{Assum}), we plot the iteration numbers with respect to some values of $\alpha \in [0,1]$ for the cases of $n=200$ and $n=2000$ in Figure \ref{fig0}. From the plotted curves, the optimality of $\alpha=0.5$ and the acceleration of other values near $0.5$ are well demonstrated.

\begin{table}[H]
\caption{Improvement of (\ref{Assum-0.75}) for CP-PPA \eqref{C-P}.}
 \centering
  \setlength{\tabcolsep}{1.8mm}{
 \begin{tabular}{l ccccccc ccc}
  \toprule
 \multirow{2}{*}{$n$} & \multirow{2}{*}{$\rho(A^T\!A)$} & \multicolumn{4}{c}{CP-PPA with (\ref{Assum-CP})} &   \multicolumn{4}{c}{CP-PPA with (\ref{Assum-0.75})} & \cr \cmidrule(lr){3-6} \cmidrule(lr){7-10}
                 &      &         It.      &  CPU     &  $\|x^k\|_1$   & $\|Ax^k-b\|$        &   It.     &  CPU   &    $\|x^k\|_1$  &  $\|Ax^k-b\|$           \cr
  \midrule
$100$       &     214.14     &    464    &    0.13    &      23.61      &   7.84e-8      &   342      &   0.06    &      23.61    &   8.98e-8    \\
$200$       &     444.51     &    714    &    0.17    &      49.88      &   6.84e-8      &   531      &   0.11    &      49.88    &   7.71e-8    \\
$300$       &     691.99     &   1065   &    0.29    &      83.25      &   2.47e-7      &   827      &   0.17    &      83.25    &   1.57e-7    \\
$400$       &     863.35     &    780    &    0.19    &     116.46     &   1.58e-7      &   596      &   0.13    &     116.46   &   1.56e-7    \\
$500$       &    1071.84    &    928    &    0.25    &     142.75     &   1.75e-7      &   703      &   0.16    &     142.75   &   1.69e-7    \\
$800$       &    1749.00    &    858    &    0.24    &     234.90     &   2.26e-7      &   647      &   0.16    &     234.90   &   2.37e-7    \\
$1000$     &    2213.51    &    948    &    0.27    &     281.97     &   3.58e-7      &   734      &   0.19    &     281.97   &   2.56e-7    \\
$2000$     &    4486.57    &   1022   &    0.92    &     507.92     &   1.16e-7      &   782      &   0.61    &     507.92   &   9.80e-8    \\
$3000$     &    6651.24    &   1103   &    2.97    &     768.48     &   2.97e-7      &   886      &   2.30    &     768.48   &   2.12e-7    \\
$5000$     &   11230.30   &   1092   &    9.62    &    1244.81    &   3.62e-7      &   869      &   7.62    &    1244.81  &   1.39e-7    \\
  \bottomrule
 \end{tabular}}
 \label{Ta1}
\end{table}

\begin{figure}[H]
\centering
\subfigure[$n=200$]{
\includegraphics[width=8.3cm]{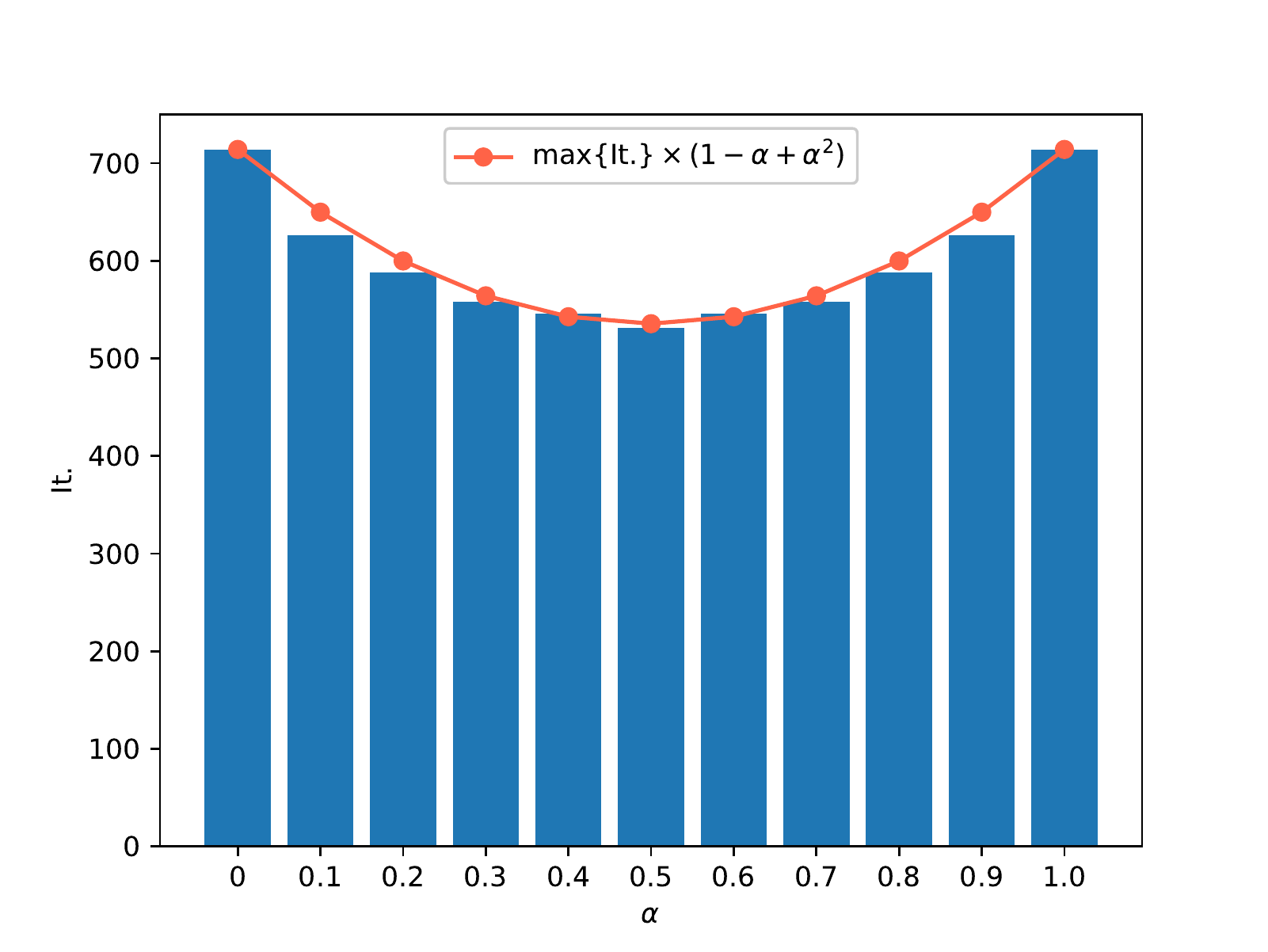}
}\hspace{-10mm}
\subfigure[$n=2000$]{
\includegraphics[width=8.3cm]{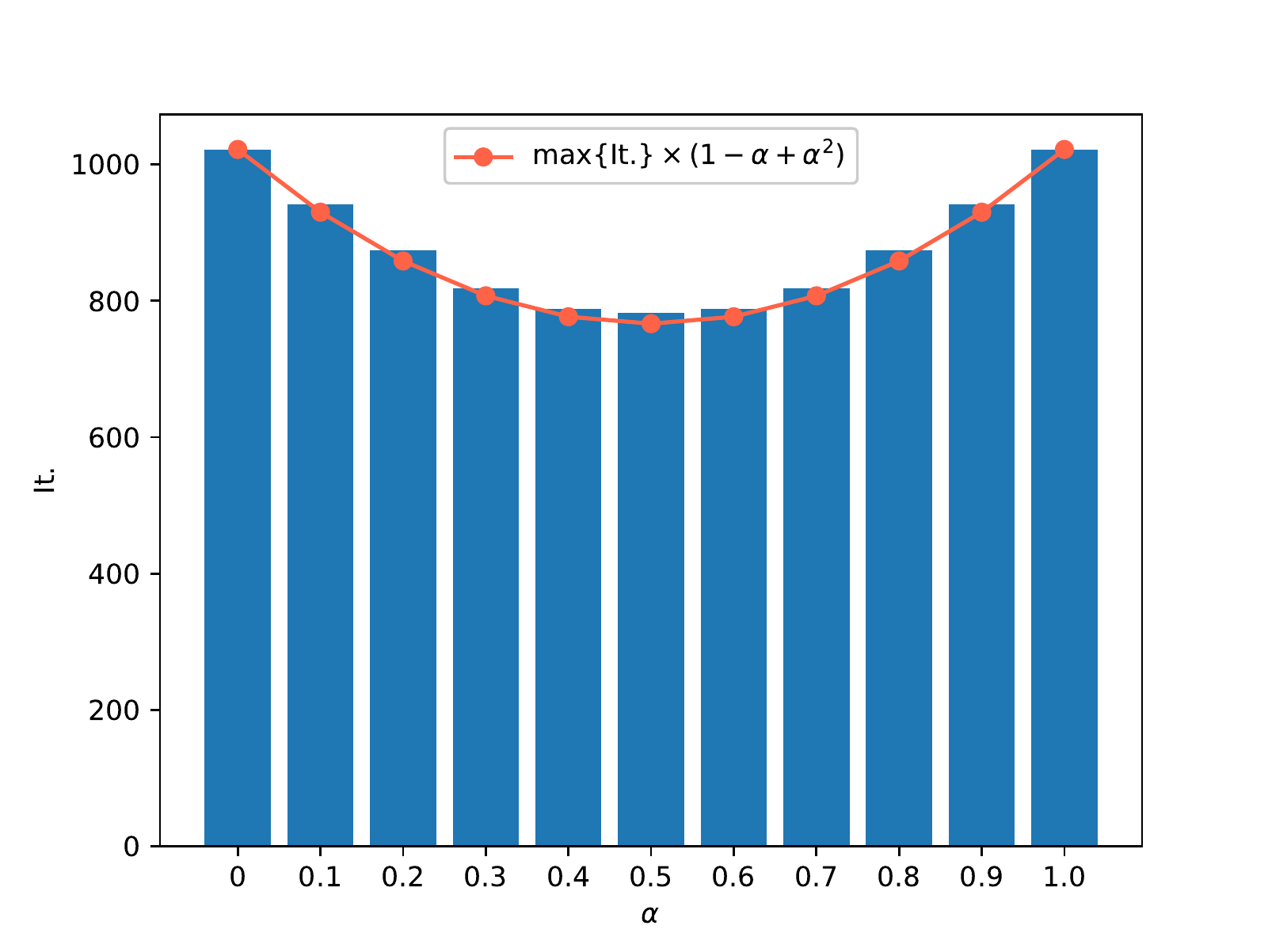}
}
\caption{Performance of CP-PPA \eqref{C-P} with different $\alpha$ in (\ref{Assum}) for \eqref{BP}. }
\label{fig0}
\end{figure}

\subsection{Potts-model-based image segmentation problem}

Then, we test the convex-relaxed version of the Potts model for multi-phase image segmentation problem
$$\min_{u_i(x)\geq0}\Big\{\sum_{i=1}^{m}\int_{\Omega}u_i(x)\rho(l_i,x)+\mu|\nabla u_i(x)|dx  \;\Big | \;\sum_{i=1}^mu_i(x)=1 \Big\},$$
where $\mu>0$ is the weight parameter for the regularization term of the total perimeter of all segmented regions, $\int_{\Omega}|\nabla u_i(x)|dx$ $(i=1,\ldots,m)$ are total variation (TV) regularization terms (see \cite{ROF1992}),  $\rho(l_i,x)$ $(i=1,\ldots,m)$ are used to evaluate the cost of assigning the label $l_i$ to the specified position $x$, and $\Omega$ is the image domain. As analyzed in \cite{yuan2010continuous},  its corresponding dual problem can be reformulated as
\begin{equation}\label{Potts}
  \begin{aligned}
& \max\int_{\Omega}p_s(x)dx\\
& \;\; \hbox{s.t. }\D q_i(x)-p_s(x)+p_i(x)=0,\; i=1,2,\ldots,m,\\
& \qquad\;  |q_i(x)|\leq \mu,\;p_i(x)\leq\rho(l_i,x),\; i=1,2,\ldots,m,
\end{aligned}
\end{equation}
where $\D=(-\nabla)^\ast$ is the adjoint of the $-\nabla$ operator, $p_s(x)$ and $p_i(x)$ $(i=1,\ldots,m)$ stand for the source flow and the sink flow, respectively. After discretization,  let $I_C(p(x),q(x))$ be the convex characteristic function on the convex set
$$C := \{(p(x);q(x))\;|\;p_i(x)\leq\rho(l_i,x),|q_i(x)|\leq \mu,\; i=1,2,\ldots,m\},$$
${\bf 1}$ be is the $n$-vector whose elements are all 1, and $I_n$ be the $n\times n$ identity matrix. Then, the model \eqref{Potts} can be rewritten as
\begin{subequations}\label{Potts-image}
\begin{equation}
\min_{p_s, p, q}\; \underbrace{- {\bf 1}^T p_s \, + \, I_C(p,q)}_{f(p_s, p,q)}
\end{equation}
subject to the following linear equality constraints
\begin{equation}
\underbrace{\left(\!\!\!
    \begin{array}{c}
      -I_n \\
      -I_n \\[-0.1cm]
      \vdots \\
      -I_n \\
    \end{array}
 \!\! \!\right)p_s+\left(\!\!\!
               \begin{array}{c}
                 \D \\
                 0 \\[-0.1cm]
                 \vdots \\
                 0 \\
               \end{array}
            \!\!\! \right)q_1+\!\cdots\!+\left(\!\!\!
                                            \begin{array}{c}
                                              0 \\
                                              0 \\[-0.1cm]
                                              \vdots \\
                                              \D  \\
                                            \end{array}
                                          \!\!\!\right)q_m+\left(\!\!\!
               \begin{array}{c}
                 I_n \\
                 0 \\[-0.1cm]
                 \vdots \\
                 0 \\
               \end{array}
            \!\!\! \right)p_1+\!\cdots\!+\left(\!\!\!
                                            \begin{array}{c}
                                              0 \\
                                              0 \\[-0.1cm]
                                              \vdots \\
                                              I_n  \\
                                            \end{array}
                                          \!\!\!\right)p_m}_{A(p_s;\;q_1;\ldots;\;q_m;\;p_1;\ldots;\;p_m)}=\underbrace{\left(\!\!\!
                                            \begin{array}{c}
                                              0 \\
                                              0 \\[-0.1cm]
                                              \vdots \\
                                              0  \\
                                            \end{array}
                                          \!\!\right)}_{b}.
\end{equation}
\end{subequations}
Hence, it is a special case of \eqref{Problem-LC} with the specific matrix $A$ defined as
$$A=\left(\!\!
    \begin{array}{ccccccccc}
      -I_n & \D &  0 &    \cdots            &      0   & I_n &  0 &    \cdots            &      0      \\
      -I_n        &     0       &         \D &   \cdots        &              0 &     0       &       I_n&   \cdots        &              0\\[-0.1cm]
      \vdots &       \vdots  &   \vdots     &        \ddots            &  \vdots &       \vdots  &   \vdots     &        \ddots            &  \vdots   \\
      -I_n        &     0       &    0     &     \ldots          &             \D  &     0       &    0     &     \ldots          &         I_n \\
    \end{array}\!\!
  \right).
$$
We thus have
$$AA^T=\left(\!\!
                                         \begin{array}{cccc}
                                           2I_n-\Delta & I_n            & \cdots & I_n \\
                                           I_n            & 2I_n-\Delta & \cdots & I_n \\[-0.1cm]
                                           \vdots    & \vdots   & \ddots & \vdots \\
                                           I_n & I_n & \cdots & 2I_n-\Delta \\
                                         \end{array}\!\!
                                       \right).
$$
As analyzed in \cite{Sun2021},  we have $\rho(A^T\!A)=\rho(AA^T)\leq9+m$ for the Potts-based image segmentation model \eqref{Potts-image}, and the CP-PPA \eqref{C-P} for \eqref{Potts-image} can be specified as
\begin{equation}\label{CP-Potts}
  \left\{
    \begin{array}{lll}
      q_i^{k+1}&=&\mathcal{P}_{\mu}(q_i^k-\frac{1}{r}\nabla u_i^k), \;i=1,\ldots,m,  \\[0.2cm]
      p_i^{k+1}&=&\mathcal{P}_{\rho_i}(p_i^k+\frac{1}{r}u_i^k), \;i=1,\ldots,m, \\[0.2cm]
      p_s^{k+1}&=&p_s^k+\frac{1}{r}(1-\sum_{i=1}^mu_i^k),\\[0.2cm]
      u_i^{k+1} &=&u_i^k-\frac{1}{s} \big\{\D (2q_i^{k+1}-q_i^k)-(2p_s^{k+1}-p_s^k)+(2p_i^{k+1}-p_i^k)\big\}, \; i=1,\ldots,m,
    \end{array}
  \right.
\end{equation}
where the projections $\mathcal{P}_{\mu}$ and $\mathcal{P}_{\rho_i}\;(i=1,\ldots,m)$ are respectively defined as follows:
$$\mathcal{P}_{\mu}(q)=q/\max\Big\{1,\frac{|q|}{\mu}\Big\} \;\; \hbox{and} \;\; \mathcal{P}_{\rho_i}(p)=\min\Big\{p,\rho(l_i,x)\Big\} \;\; \hbox{for} \;\; i=1,\ldots,m.$$

For succinctness, we only consider the case of $m=4$ for the Potts-based image segmentation model  \eqref{Potts-image}. Note that $\rho(A^T\!A)\leq9+m=13$ in \eqref{Potts-image}. We simply follow \cite{Sun2021}, and choose $r\cdot s=13$ to satisfy the condition (\ref{Assum-CP}) (which is numerically reasonable because the condition (\ref{Assum-CP}) only requires that $r\cdot s$ be arbitrarily larger than $\rho(A^TA)$). Accordingly, the optimal condition (\ref{Assum-0.75}) becomes $r\cdot s=0.75\times13$.  Again, for comparison purpose, we use the same mechanism to choose $r$ and $s$ for implementing \eqref{CP-Potts}, and they are chosen in the following ways subject to the only difference of the constant $0.75$:
\begin{itemize}
  \item CP-PPA with (\ref{Assum-CP}): $r=13/s$  and  $s=3,4,5,6,7$.
  \item CP-PPA with (\ref{Assum-0.75}): $r=0.75\times13/s$ and $s=3,4,5,6,7$.
\end{itemize}

We follow the package developed by the authors of \cite{yuan2010continuous} (which is available at \url{https://www.mathworks.com/matlabcentral/fileexchange/34224-fast-continuous-max-flow-algorithm-to-2d-3d-multi-region-image-segmentation}) and take the stopping criterion for \eqref{CP-Potts} as
$$\hbox{ADE}(k):=\frac{\|u^k-u^{k-1}\|}{\hbox{size}(u)}<10^{-7},$$
where ``ADE" denotes the average dual error. For experiments, we test the image ``flowers" with size $539\times359$ which can be downloaded from \url{https://github.com/taigw/GrabCut-GraphCut}, and the image ``butterfly" with size $768\times512$ can be downloaded from \url{https://homepages.cae.wisc.edu/~ece533/images/}.

\begin{table}[H]
\caption{Numerical results of CP-PPA \eqref{C-P} for \eqref{Potts-image} with $m=4$.}
 \centering
 \begin{tabular}{l ccccccc ccc}
  \toprule
 \multirow{2}{*}{Parameter $s$} &   \multicolumn{3}{c}{CP-PPA with (\ref{Assum-CP})} &   \multicolumn{3}{c}{CP-PPA with (\ref{Assum-0.75})}& \cr \cmidrule(lr){2-4} \cmidrule(lr){5-7}
                                &   It.     &   CPU     &   ADE     &   It.   &    CPU      &   ADE      \cr
  \midrule
   flowers       \\\cline{1-1}\\[-0.3cm]
   $s=3$      &  426       &    78.74       &  9.99e-8      &    335       &     62.19         &   9.99e-8   \\
   $s=4$      &  422       &    78.54       &  9.97e-8      &    336       &     61.68         &   9.93e-8   \\
   $s=5$      &  419       &    78.30       &  9.99e-8      &    332       &     61.78         &   9.94e-8   \\
   $s=6$      &  412       &    76.47       &  9.91e-8      &    329       &     60.30         &   9.89e-8   \\
   $s=7$      &  408       &    75.45       &  9.86e-8      &    318       &     58.39         &   9.89e-8   \\
   \midrule
    butterfly       \\\cline{1-1}\\[-0.3cm]
   $s=3$      &  572       &    194.13       &  9.99e-8      &    480       &     162.97        &   9.91e-8   \\
   $s=4$      &  529       &    181.25       &  9.91e-8      &    450       &     154.28        &   9.87e-8   \\
   $s=5$      &  503       &    172.59       &  9.96e-8      &    433       &     147.86        &   9.85e-8   \\
   $s=6$      &  488       &    167.92       &  9.99e-8      &    422       &     143.75        &   9.96e-8   \\
   $s=7$      &  482       &    164.87       &  9.95e-8      &    423       &     144.63        &   9.99e-8   \\
  \bottomrule
 \end{tabular}
 \label{Table2}
\end{table}

Some numerical results of the CP-PPA \eqref{C-P} for solving the Potts-based image segmentation model \eqref{Potts-image} are reported in Table \ref{Table2}. The moderate acceleration of the optimal condition (\ref{Assum-0.75}) over the original condition (\ref{Assum-CP}) is shown again for the CP-PPA (\ref{C-P}). In Figure \ref{fig1}, some tested images and the computationally segmented results are visualized.

\begin{figure}[H]
\centering
\subfigure[Original image: flowers]{
\includegraphics[width=7.5cm]{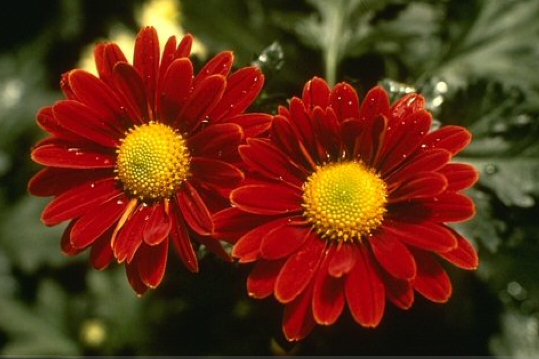}
}\hspace{5mm}
\subfigure[Segmented result of flowers]{
\includegraphics[width=7.5cm]{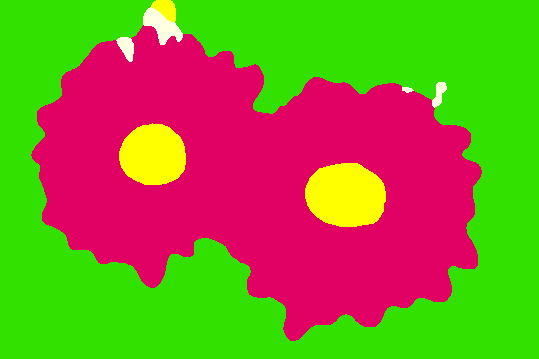}
}\\[0.2cm]
\subfigure[Original image: butterfly]{
\includegraphics[width=7.5cm]{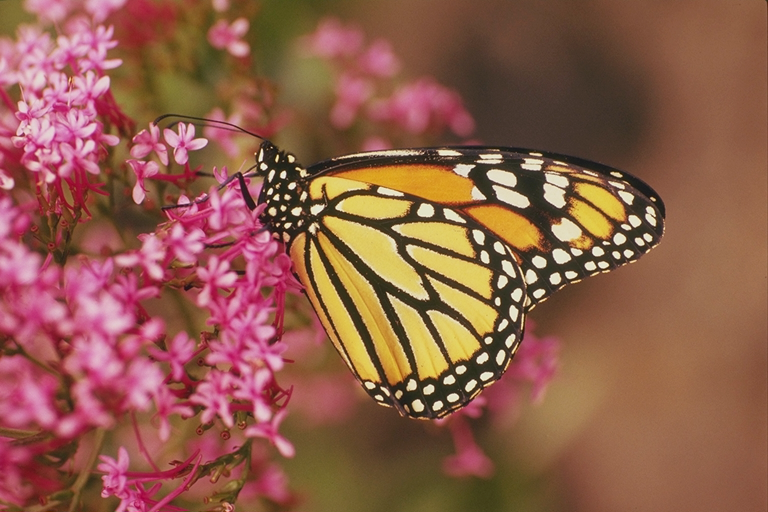}
}\hspace{5mm}
\subfigure[Segmented result of butterfly]{
\includegraphics[width=7.5cm]{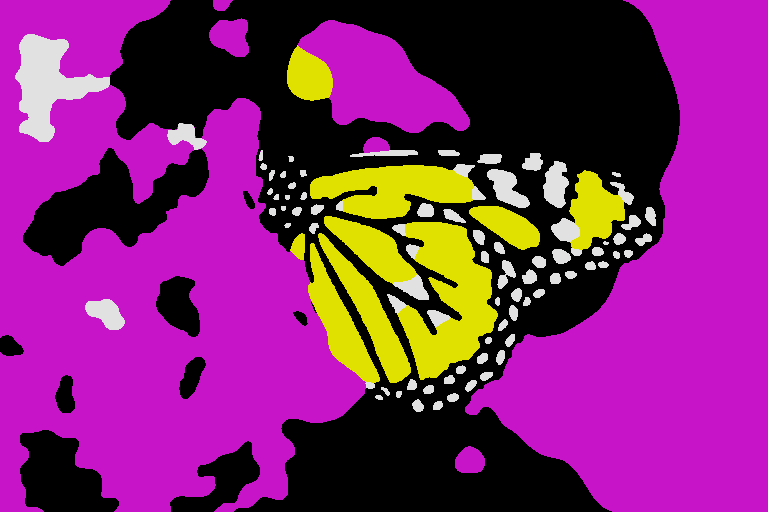}
}
\caption{Clean and computationally segmented images by CP-PPA (\ref{C-P}) with (\ref{Assum-0.75}) for \eqref{Potts-image} with four labels; $s=7$ for ``flowers" and  $s=6$ for ``butterfly".}
\label{fig1}
\end{figure}

\section{Heuristic for improving (\ref{Assum-CP})}\label{sec-exp}

\setcounter{equation}{0}
\setcounter{remark}{0}

The convergence-guaranteeing condition (\ref{Assum}) is less restrictive than the original condition (\ref{Assum-CP}), and its optimal choice (\ref{Assum-0.75}) with $\alpha=0.5$ is verified to be effective for accelerating the CP-PPA (\ref{C-P}) moderately. This improvement is effective for the generic case where the corresponding $A^T\!A$ is not assumed to have any specific structure. As mentioned, if $\rho(A^T\!A)$ is too large, then the condition (\ref{Assum-CP}) restricts the choices of $r\cdot s$ severely and tiny step sizes become unavoidable. For this case, the condition \eqref{Assum-CP} becomes the bottleneck to the efficiency of the CP-PPA \eqref{C-P}, and the improved condition (\ref{Assum}) or (\ref{Assum-0.75}) does not help much, either.

\subsection{Motivation}

When $\rho(A^T\!A)$ is too large, it is necessary to consider how to further relax these conditions and thus to avoid too small step sizes, even though rigorous convergence may not be guaranteed. Note that the essential role of the conditions (\ref{Assum-CP}), (\ref{Assum}) and (\ref{Assum-0.75}) in the convergence proof of the CP-PPA \eqref{C-P} is to sufficiently hence most conservatively guarantee the inequalities
\begin{equation}\label{Assum1}
  r\cdot s \|x^k - {x}^{k+1} \|^2 > \|A(x^k - {x}^{k+1})\|^2
\end{equation}
for all iterations uniformly. In this sense, these conditions provide iteration-independent and hence globally applicable bounds for the choice of $r\cdot s$. Indeed, all these conditions are for the generic setting of (\ref{Min-Max}) and no structure of $A^TA$ is pre-assumed. Our essential idea is to consider the less restrictive condition (\ref{Assum1}) directly, and explore the structure of $A$ to probe smaller lower bound of $r\cdot s$ iteratively in lieu with the iteration-dependent condition (\ref{Assum1}) directly. If the average eigenvalue of $A^T\!A$ (denoted by $\rho_{\hbox{\footnotesize average}}(A^T\!A)$) is significantly smaller than the largest one, then we are inspired to determine the range of $r\cdot s$ heuristically in terms of $\rho_{\hbox{\footnotesize average}}(A^T\!A)$, instead of $\rho(A^TA)$, and we even allow for certain flexibility in the requirement of satisfying the condition (\ref{Assum1}) strictly. Though there is no mathematical rigor, the condition (\ref{Assum1}) may be satisfied for some iterations. Once it is satisfied, tiny step sizes are avoided and hence much computation can be saved. This is how the niche targeting heuristic works for the CP-PPA (\ref{C-P}).

\subsection{Assignment problem}\label{sec-assignment}

Recall that the classic assignment problem in operational research aims at assigning $n$ jobs to $n$ persons with each job being exactly assigned to one person, and its model is
  \[  \label{Assignment}
   \begin{array}{rl}   \max & \Phi(x):=\sum_{i=1}^n  \sum_{j=1}^n   c_{ij} x_{ij}   \\
                                    \hbox{s.t}  &   \sum_{j=1}^n   x_{ij}  = 1, \quad i=1, \ldots, n,\\
                                              &   \sum_{i=1}^n   x_{ij}  = 1,  \quad j=1, \ldots, n,  \\
                                            &     x_{ij} \in \{ 0, 1\},
                                    \end{array}
                                                                        \]
where $c_{ij}>0$. It is well known (e.g. Section 6.5 in \cite{LYe}) that the assignment problem (\ref{Assignment}) can be relaxed as the following one in which the binary constraints are replaced by box constraints:
  \[  \label{Assignment-R}
   \begin{array}{rl}   \min &-\Phi(x):=\sum_{i=1}^n  \sum_{j=1}^n   -c_{ij} x_{ij}   \\
                                    \hbox{s.t}  &   \sum_{j=1}^n   x_{ij}  = 1,  \quad i=1, \ldots, n,  \\
                                              &   \sum_{i=1}^n   x_{ij}  = 1,   \quad j=1, \ldots, n, \\
                                              &   0 \leq x_{ij} \leq 1.
                                    \end{array}
                                    \]
The model (\ref{Assignment-R}) is obviously a special case of the model \eqref{Problem-LC} with a linear objective function, linear equality constraints, and box constraints. Now we focus on finding an efficient heuristic to further accelerate the CP-PPA \eqref{C-P} by exploiting the special structure of the specific problem (\ref{Assignment-R}).

Let us define
   $$    x^T= (x_{11}, x_{12}, \ldots, x_{1n}, x_{21}, x_{22}, \ldots, x_{2n}, \cdots\cdots , x_{n1}, x_{n2}, \ldots, x_{nn}).
       $$
Then, for the model (\ref{Assignment-R}), the matrix $A$ corresponding to (\ref{Problem-LC}) has the form
\[  \label{Matrix-A}
    A=\left( \begin{array}{ccccc}
                   \quad e^T  \quad   &            &                             &               &      \\[-0.2cm]
                               &   \quad e^T  \quad   &                             &                &       \\[-0.2cm]
                               &            &  \quad  \cdots \cdots  \quad   &                &       \\[-0.2cm]
                              &            &                              &   \quad  e^T  \quad     &      \\[-0.2cm]
                                &          &                            &                    &     \quad e^T   \quad       \\
             \hbox{\large $ I_n$}  &  \hbox{\large $ I_n$}  &      \cdots   \cdots   & \hbox{\large $ I_n$}   &  \hbox{\large $ I_n$}
              \end{array} \right),
             \]
where $e$ is the $n$-vector whose elements are all 1, and $I_n$ is the $n\times n$ identity matrix. Notice that the $2n \times n^2$ matrix $A$ defined in (\ref{Matrix-A}) is totally unimodular, and as analyzed in, e.g., \cite{Sch1998}, this property essentially ensures the equivalence between the assignment problem (\ref{Assignment}) and its relaxed one (\ref{Assignment-R}).

Now, we explain how to calculate the largest and average eigenvalues of the matrix $A^T\!A$. First, it follows from basic linear algebra knowledge that $ \rho(A^T\!A) =  \rho(AA^T)$ and $ \hbox{Trace}(A^T\!A)    = \hbox{Trace}(AA^T)$. Using the structure of the matrix $A$ in (\ref{Matrix-A}), we have
\[\label{AAT}
 AA^T  =
    \left(\begin{array}{cccccccccc}
                 n     & 0        & \ldots   &   \ldots  &   0           &          1   &  1   & \ldots &\ldots   & 1       \\[-0.1cm]
                 0     &  n       & \ddots  &                &  \vdots   &          1   &  1   & \ldots & \ldots   & 1      \\[-0.1cm]
           \vdots &\ddots &\ddots   &  \ddots  & \vdots    &  \vdots  &\vdots  &  &   & \vdots            \\[-0.1cm]
           \vdots &             &\ddots   &  \ddots  &        0      &   1 &  1   & \ldots & \ldots   & 1               \\[-0.1cm]
                0    &  \ldots &\ldots    &       0       &        n      &   1 &  1   & \ldots & \ldots   & 1               \\[-0.1cm]
                  1   &  1   & \ldots &\ldots   & 1    &     n     & 0        & \ldots   &   \ldots  &   0        \\[-0.1cm]
                            1   &  1   & \ldots & \ldots   & 1   & 0     &  n       & \ddots  &                &  \vdots           \\[-0.1cm]
                   \vdots  &\vdots  &  &   & \vdots   &   \vdots &\ddots &\ddots   &  \ddots  & \vdots           \\[-0.1cm]
                     1 &  1   & \ldots & \ldots   & 1   &  \vdots &             &\ddots   &  \ddots  &        0                \\[-0.1cm]
                      1 &  1   & \ldots & \ldots   & 1     &    0    &  \ldots &\ldots    &       0       &        n
                  \end{array} \right)            =
                      \left(\begin{array}{cc}
                        nI_n     & ee^T  \\
                        ee^T      &    nI_n  \end{array} \right) .
        \]
According to the assertion of the Gerschgorin-circle for the eigenvalues of a matrix, for every eigenvalue $\rho_i$ of the positive semi-definite matrix $AA^T$, we have
        $$     | \rho_i- n |  \le  n   \qquad \hbox{and thus} \qquad    0\le \rho_i \le 2n.  $$
In addition,  we have
$$     AA^T   \left(\begin{array}{c}
                           e  \\
                          e \end{array} \right) =\left(\begin{array}{cc}
                        nI_n     & ee^T  \\
                        ee^T      &    nI_n  \end{array} \right)  \left(\begin{array}{c}
                           e  \\
                          e \end{array} \right)  = 2n \left(\begin{array}{c}
                           e  \\
                          e \end{array} \right). $$
Therefore, it holds that
  $$    \rho(A^T\!A) = \rho(AA^T)   = 2n .  $$
On the other hand, it follows from \eqref{AAT} that
$$  \hbox{Trace}(A^T\!A)= \hbox{Trace}(AA^T)  = 2n^2, $$
and thus $\rho_{\hbox{\footnotesize average}}(A^T\!A) = 2$. That is, for the matrix $A$ defined in (\ref{Matrix-A}) $\rho(A^T\!A)=2n$ while $\rho_{\hbox{\footnotesize average}}(A^T\!A) \equiv 2$. When $n$ is large, the condition (\ref{Assum-CP}), as well as its improved ones (\ref{Assum}) and (\ref{Assum-0.75}), are all too conservative to yield favorable step sizes, because of $\rho(A^T\!A)=2n$. To avoid tiny step sizes for this case, the significant difference between $\rho(A^T\!A)$ and $\rho_{\hbox{\footnotesize average}}(A^T\!A)$ naturally inspires us to consider replacing these conditions heuristically by
\begin{equation}\label{Assum-Heur}
  r \cdot s = 2\rho_{\hbox{\footnotesize average}}(A^T\!A),
\end{equation}
to implement the CP-PPA (\ref{C-P}) for the model (\ref{Assignment-R}).

To simulate, we generate some data for the model \eqref{Assignment-R} with different values of $n$. We take $ c_{ij} = \hbox{random} \times  10$ for $ i=1, \ldots, n,  j=1, \ldots, n$;
    $x_{ij}^0 = \frac{1}{n}$ for $i=1,\ldots, n, \; j=1, \ldots, n$; and $y^0= \textbf{0}$. The stopping criterion is
$$\max\big\{\|x^k-x^{k-1}\|_\infty,\quad \|y^k-y^{k-1}\|_\infty\big\}<10^{-10}.$$
We still include the improved condition (\ref{Assum-0.75}) for comparison, although our main purpose is to verify the effectiveness of the heuristic condition (\ref{Assum-Heur}). More specifically, we compare the following specific choices of $r$ and $s$ to implement the CP-PPA \eqref{C-P}.
\begin{itemize}
\item  CP-PPA with \eqref{Assum-CP}: $r=(10/n) \times \sqrt{n/2}$ and $s = 0.4 n \times  \sqrt{n/2}$.
\item  CP-PPA with \eqref{Assum-0.75}: $r=\sqrt{0.75} \times (10/n) \times \sqrt{n/2}$ and $s=  \sqrt{0.75} \times 0.4 n \times  \sqrt{n/2}$.
\item  CP-PPA with \eqref{Assum-Heur}: $r=10/n$ and $s= 4/r= 0.4 n$.
 \end{itemize}

We report some numerical results for the CP-PPA \eqref{C-P} with different conditions in Table \ref{Table3}. It can be seen that the structure-exploiting heuristic (\ref{Assum-Heur}) accelerates the CP-PPA \eqref{C-P} significantly. Indeed, the acceleration of (\ref{Assum-Heur}) over the conditions (\ref{Assum-CP}) and (\ref{Assum-0.75}) can be as much as more than 20 times. Hence, the possible loss of mathematical rigor in (\ref{Assum-Heur}) leads to a significant acceleration empirically for the CP-PPA (\ref{C-P}). In addition, for all the cases tested, it is empirically verified that solutions of the relaxed problem \eqref{Assignment-R} are all binary. Hence, the CP-PPA \eqref{C-P} with the heuristical condition (\ref{Assum-Heur}) provides a very efficient solver to the challenging assignment problem \eqref{Assignment}. We visualize the solutions for some tested scenarios in Figure \ref{fig2} below.

\begin{figure}[H]
\centering
\subfigure[$n=20$]{
\includegraphics[width=5.3cm]{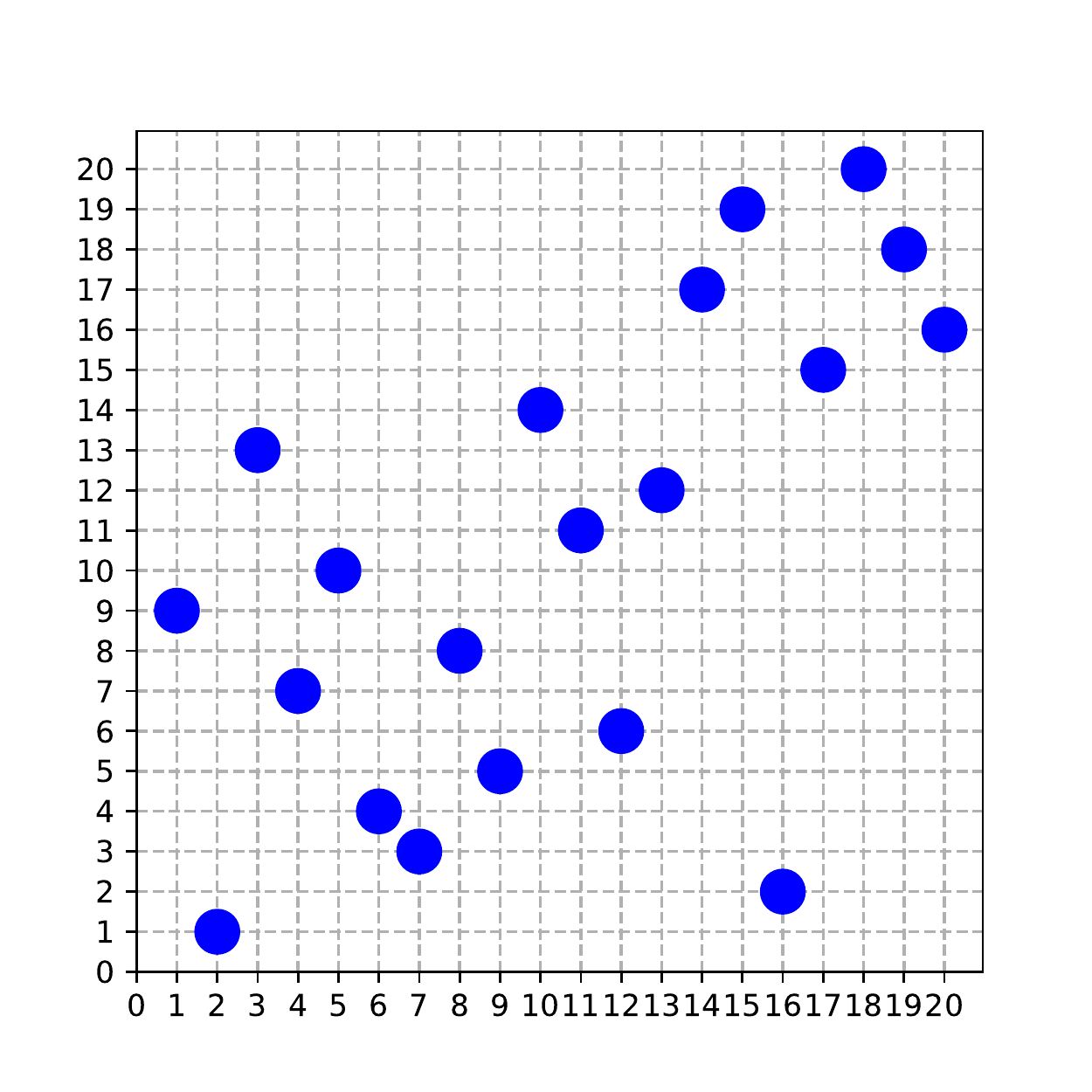}
}\hspace{-8mm}
\subfigure[$n=50$]{
\includegraphics[width=5.3cm]{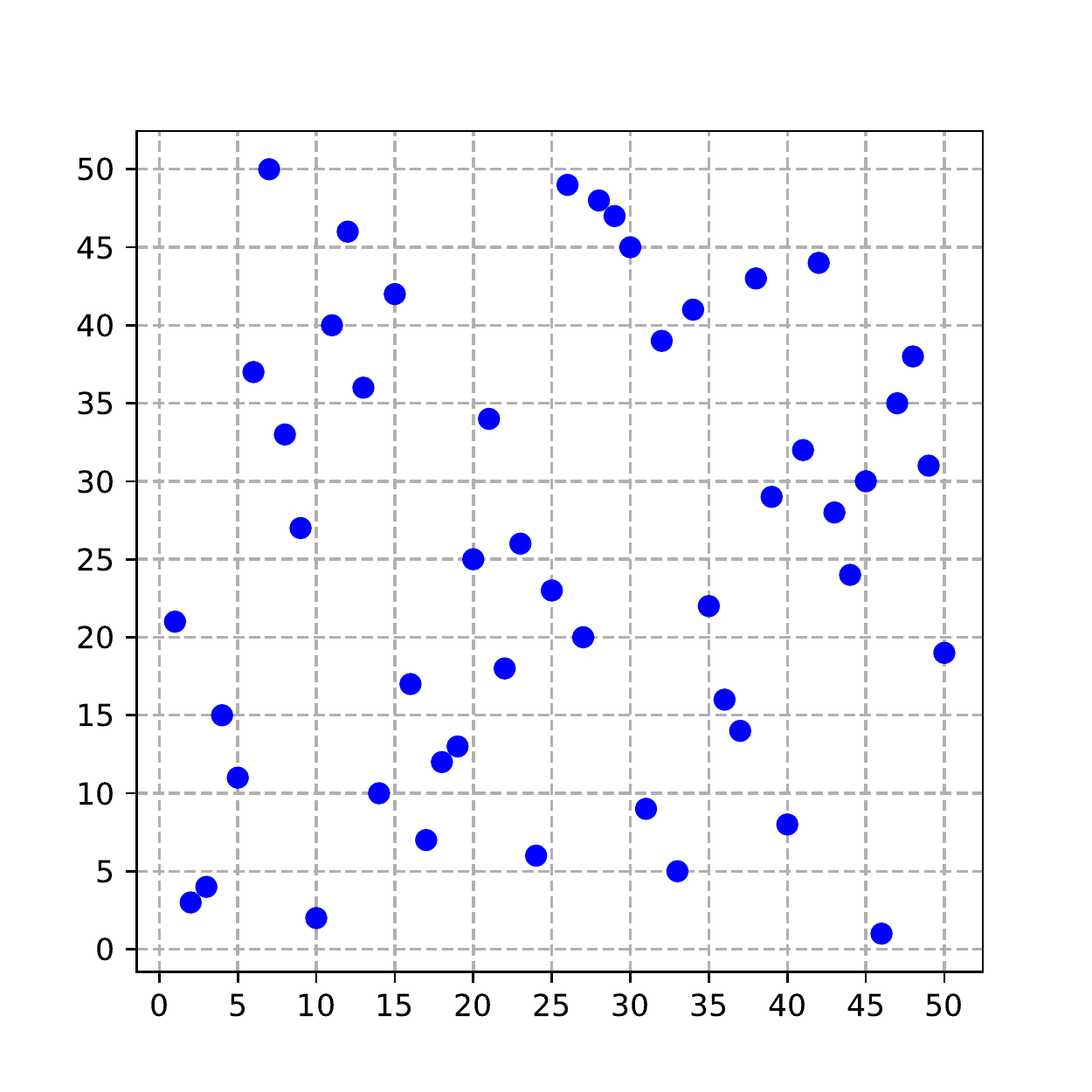}
}\hspace{-8mm}
\subfigure[$n=100$]{
\includegraphics[width=5.3cm]{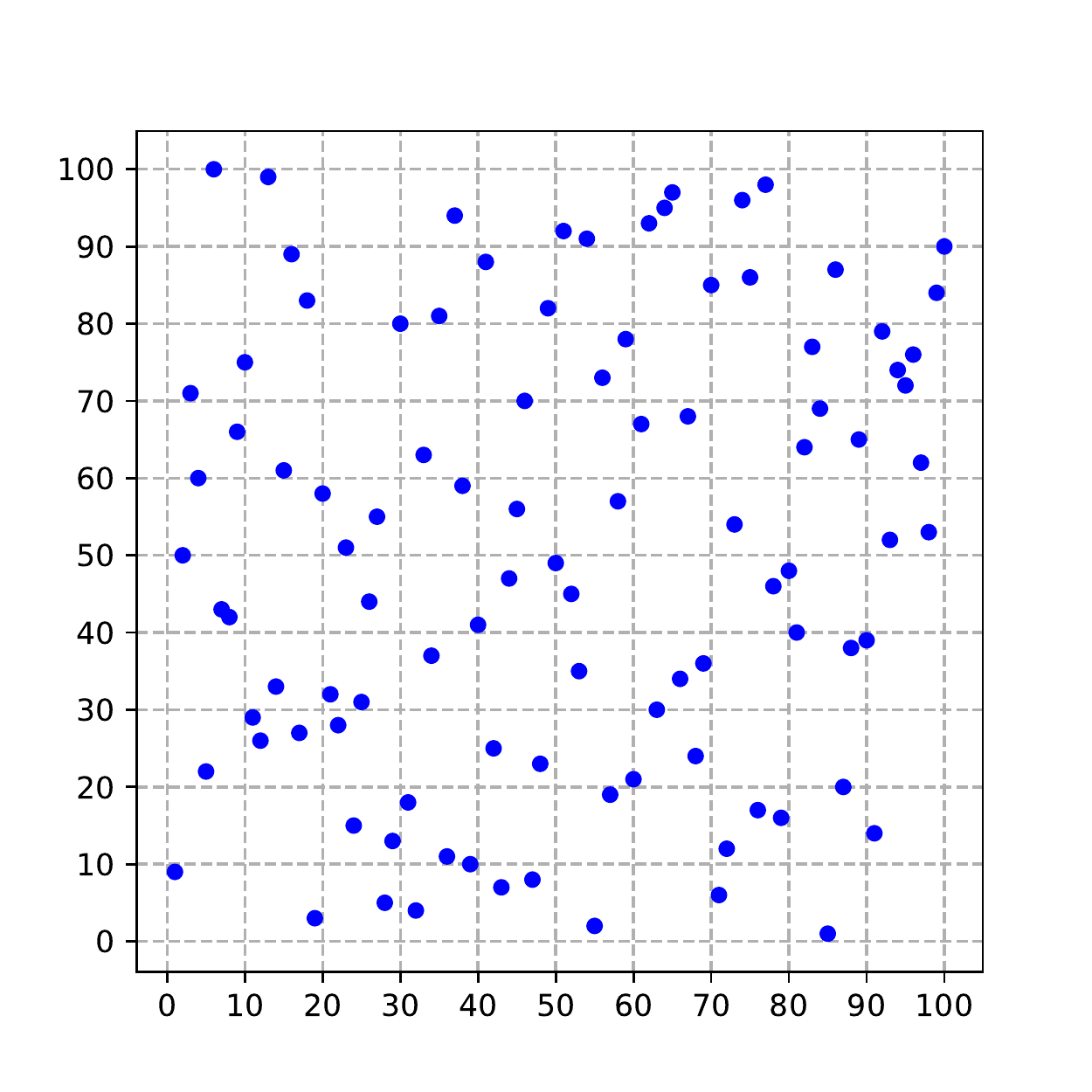}
}\\[0.0cm]
\subfigure[$n=500$]{
\includegraphics[width=8.2cm]{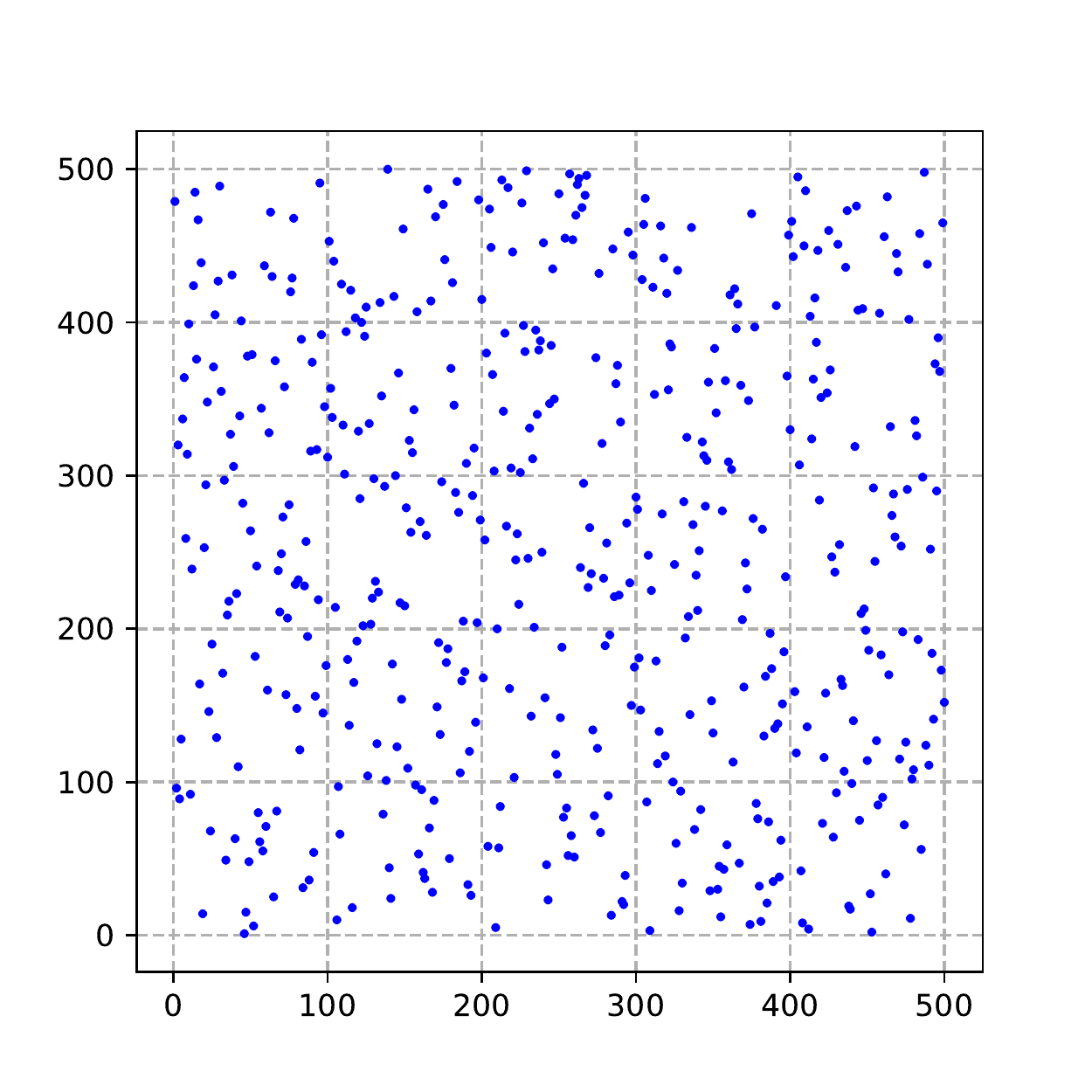}
}\hspace{-10mm}
\subfigure[$n=1000$]{
\includegraphics[width=8.2cm]{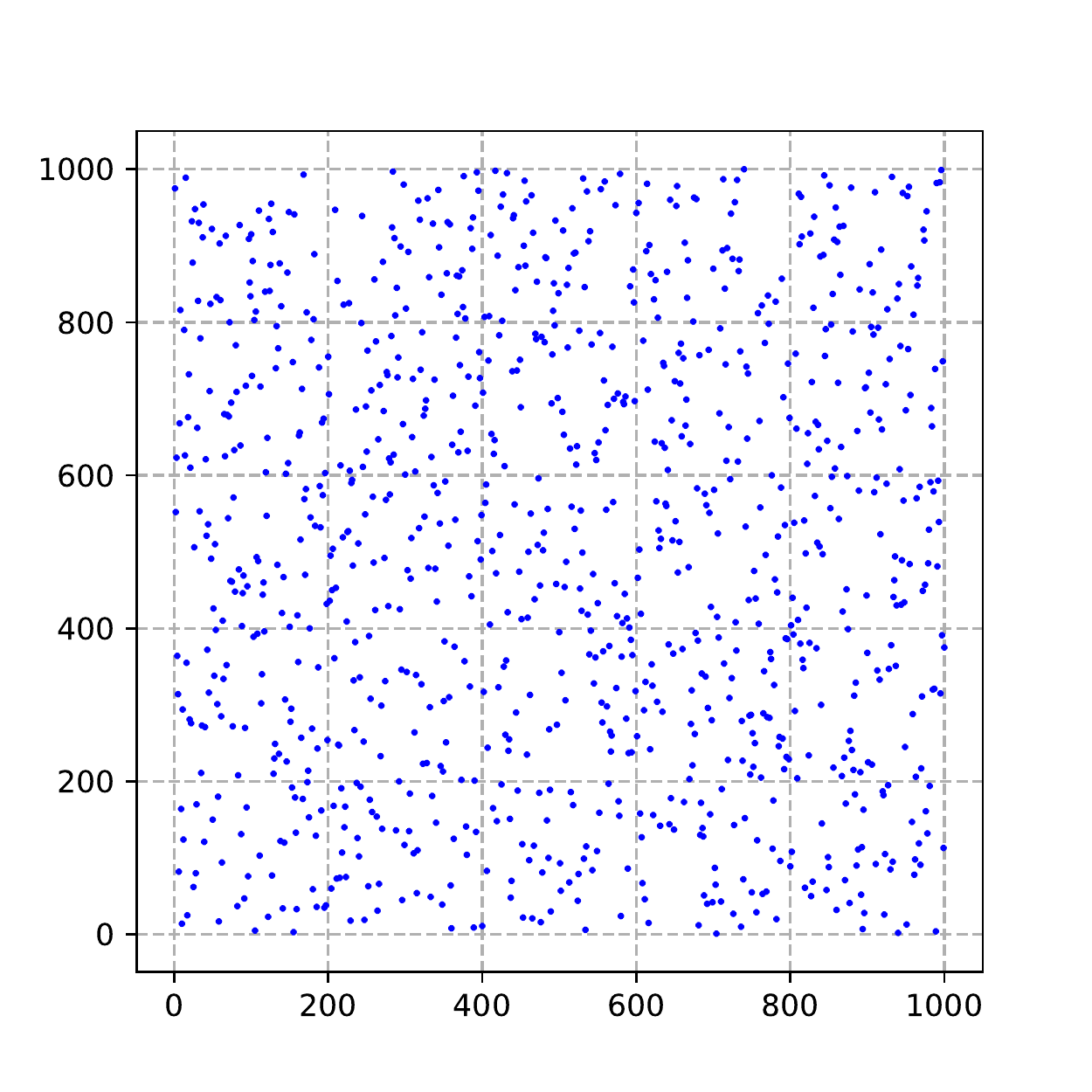}
}
\caption{Numerical results for (\ref{Assignment-R}) solved by CP-PPA \eqref{C-P} with the heuristic condition \eqref{Assum-Heur}.}
\label{fig2}
\end{figure}

\begin{table}[H]
\caption{Numerical results of  the CP-PPA \eqref{C-P} for \eqref{Assignment-R}. }
 \centering
 \setlength{\tabcolsep}{1.6mm}{
 \begin{tabular}{c ccccccc ccc}
  \toprule
 \multirow{2}{*}{Size $n$} &   \multicolumn{3}{c}{CP-PPA with \eqref{Assum-CP}} &   \multicolumn{3}{c}{CP-PPA with \eqref{Assum-0.75}} &   \multicolumn{3}{c}{CP-PPA with \eqref{Assum-Heur}} & \cr \cmidrule(lr){2-4} \cmidrule(lr){5-7} \cmidrule(lr){8-10}
                                &      It.          &   CPU    &   $\Phi(x^k)$      &     It.           &    CPU   &   $\Phi(x^k)$     &   It.       &   CPU &    $\Phi(x^k)$    \cr
  \midrule
                $20$        &    229         &    0.06        &   186.64      &    201         &    0.03        &   186.64     &   74       &    0.01    &    186.64     \\
                $50$        &    317         &    0.10        &   486.51      &    275         &    0.08        &   486.51     &   57       &    0.01    &    486.51     \\
                $80$        &    438         &    0.12        &   782.31      &    381         &    0.10        &   782.31     &   69       &    0.02    &    782.31     \\
                $100$      &    1053       &    0.35        &   983.74      &    910         &    0.28        &   983.74     &   144     &    0.05    &    983.74     \\
                $200$      &    2139       &    1.53        &   1983.47    &    1841       &    1.27        &   1983.47   &   200     &    0.14    &    1983.47     \\
                $300$      &    2407       &    3.82        &   2983.28    &    2081       &    3.20        &   2983.28   &   217     &    0.33    &    2983.28     \\
                $400$      &    10330     &    76.77      &   3982.58    &    8886       &    64.82      &   3982.58   &   699     &    5.29    &    3982.58     \\
                $500$      &    4739       &    53.62      &   4983.45    &    4070       &    45.76      &   4983.45   &   296     &    3.43    &    4983.45     \\
                $600$      &    9321       &    151.68    &   5983.80    &    8049       &    130.17    &   5983.80   &   505     &    8.14    &    5983.80     \\
                $700$      &    18774     &    413.38    &   6984.00    &    16280     &    358.55    &   6984.00   &   1041   &    23.03  &    6984.00     \\
                $800$      &    20660     &    597.20    &   7983.59    &    17816     &    517.03    &   7983.59   &   1034   &    30.16  &    7983.59     \\
                $900$      &    9500       &    349.11    &   8983.77    &    8240       &    300.18    &   8983.77   &   443     &    16.52  &    8983.77     \\
                $1000$    &    16699     &    759.72    &   9984.04    &    14466     &    659.46    &   9984.04   &   704     &    31.85  &    9984.04     \\
  \bottomrule
 \end{tabular}}
 \label{Table3}
\end{table}

\section{Conclusions}\label{Sec-conclu}

\setcounter{equation}{0}
\setcounter{remark}{0}

In this paper, we generalize the well-known primal-dual algorithm proposed by Chambolle and Pock for saddle point problems, and improve the condition for ensuring its convergence. The improved condition is effective for the most generic setting of convex programming problems with linear equality constraints, and it is shown to be optimal. We particularly recommend the choice of $\alpha=1/2$ to implement the generalized primal-dual algorithm (\ref{H-Y}) for various applications of the saddle point problem (\ref{Min-Max}), because of the resulting optimal convergence-guaranteeing condition (\ref{Assum-0.75}). It is also verified by some standard applications that the improved condition can easily lead to numerical acceleration. We would reiterate that the acceleration is generally moderate because it is resulted by a theoretically optimal condition and it is effective for the generic setting. But it is extremely easy to realize the acceleration;  essentially there is no need to tune any parameter additionally and basically we just need to adapt the available codes by attaching the constant 0.75 to the current way of determining the parameters $r$ and $s$. Hence, the improved condition immediately provides a simple and universal way to further improve the numerical performance of the original primal-dual algorithm proposed by Chambolle and Pock.
We also propose a structure-exploiting  heuristic to further accelerate the original primal-dual algorithm empirically for some specific saddle point problems, and verify its significant acceleration by the classic assignment problem.

\end{CJK*}
\end{document}